\documentclass[12pt]{amsart}
\usepackage{amssymb}
\usepackage{graphicx}
\usepackage{xcolor} 
\usepackage{tensor}
\usepackage{fullpage} 

\definecolor{green}{rgb}{0,0.8,0} 

\newtheorem{theorem}{Theorem}[section]

\newtheorem{lemma}[theorem]{Lemma}
\newtheorem{proposition}[theorem]{Proposition}
\theoremstyle{definition}

\theoremstyle{remark}
\newtheorem{remark}[theorem]{Remark}
\numberwithin{equation}{section}

\def\eps{\varepsilon }

\newcommand{\pa}{\partial}
\newcommand{\na}{\nabla}

\newcommand{\Del}{\Delta}

\newcommand{\vep}{\varepsilon}

\newcommand{\bbr}{\mathbb R}

\newcommand{\beq}{\begin{equation}}
\newcommand{\eeq}{\end{equation}}

\newcommand{\be}{\begin{equation}}
\newcommand{\ee}{\end{equation}}
\newcommand{\ben}{\begin{equation*}}
\newcommand{\een}{\end{equation*}}
\newcommand{\la}{\label}
 \newcommand{\bmat}{\begin{pmatrix}} 
  \newcommand{\emat}{\end{pmatrix}}
    \newcommand{\tilu}{{\tilde{U}}}
     \newcommand{\hatu}{{\hat{U}}}
     \newcommand{\hatn}{{\hat{n}}}
        \newcommand{\tilq}{{\tilde{q}}}
            \newcommand{\tiln}{{\tilde{n}}}
      \newcommand{\s}{\sigma}

\usepackage{verbatim}
\begin{document}
\bibliographystyle{plain}

\title{Global Well-posedness of large perturbations of traveling waves in a hyperbolic-parabolic system arising from a chemotaxis model}


\date{\today}%

\author[Choi]{Kyudong Choi}
\address[Kyudong Choi]{\newline Department of Mathematical Sciences, \newline Ulsan National Institute of Science and Technology, Ulsan 44919, Republic of Korea}
\email{kchoi@unist.ac.kr}

\author[Kang]{Moon-Jin Kang}
\address[Moon-Jin Kang]{\newline Department of Mathematics \& Research Institute of Natural Sciences, \newline Sookmyung Women's University, Seoul 04310,  Republic of Korea}
\email{moonjinkang@sookmyung.ac.kr}

\author[Vasseur]{Alexis F. Vasseur}
\address[Alexis F. Vasseur]{\newline Department of Mathematics, \newline The University of Texas at Austin, Austin, TX 78712, USA}
\email{vasseur@math.utexas.edu}

\thanks{\textbf{Acknowledgement.}  The work of KC was partially supported by NRF-2018R1D1A1B07043065 and by the POSCO Science Fellowship of POSCO TJ Park Foundation. 
The work of MK was partially supported by NRF-2019R1C1C1009355.
The work of AV was partially supported by the NSF grant: DMS 1614918. }

\begin{abstract}
  We consider a one-dimensional  system arising from a chemotaxis model in tumour angiogenesis, which is described by a Keller-Segel equation with singular sensitivity. This hyperbolic-parabolic system  is known to allow viscous shocks (so-called traveling waves), and in literature, their nonlinear stabilities have been considered   in the class of certain mean-zero small perturbations.
We show the global existence of the solution without assuming the mean-zero condition  for any  initial data as arbitrarily large perturbations around traveling waves  in the Sobolev space $H^1$  while the shock strength is assumed to be small enough.  
The main novelty of this paper is to develop the global well-posedness of any large $H^1$-perturbations of traveling wave connecting two different end states. The discrepancy of the end states is linked to the complexity of the corresponding flux, which requires a new type of an energy estimate.  
To overcome, we use the {\it a priori} contraction estimate of a weighted relative entropy functional  up to a translation, which was proved by Choi-Kang-Kwon-Vasseur  \cite{ckkv2019}. The boundedness of the shift  implies {\it a priori} bound of the relative entropy functional without a shift on any time interval of existence, which  produces a $H^1$-estimate thanks to a De Giorgi type lemma.
Moreover, to remove possibility of vacuum appearance, we use the lemma again. 
  
\end{abstract}

\maketitle

\tableofcontents
\section{Introduction and main theorem}

We consider the following one dimensional  system: 
\begin{align}\label{nq_nu} \begin{aligned}
\pa_t n -\pa_x(n q)& =\nu  \pa_{xx} n,\\
\pa_t q -\pa_xn & = 0 \quad \mbox{ for } x\in\mathbb{R} \quad\mbox{and for }t>0 
\end{aligned} \end{align}  where $\nu>0$ is a positive constant. This hyperbolic-parabolic system is closely related to  a certain Keller-Segel system (see Subsection \ref{fromchemo}).
We are interested in the global-in-time existence issue of large perturbations of traveling waves (or viscous shocks) of the above system \eqref{nq_nu}.

\subsection{Traveling waves of  \eqref{nq_nu}}
By \cite{WaHi} (also see \cite{LiWa_siam}, or see  \cite[Lemma 2.1]{ckkv2019}), it has been known that for any $\nu>0$,  \eqref{nq_nu} admits a smooth monotone traveling wave $ \tilu(x-\sigma t)=\bmat{\tiln (x-\sigma t)}\\ {\tilq (x-\sigma t)}\emat$ connecting two end-states $(n_-,q_-)$, $(n_+,q_+)\in \bbr^+\times\bbr$, \textit{i.e.}, \begin{equation}\label{bdry_cond}  {\tiln}(-\infty)=n_->0, \, \,\, \,{\tiln}(+\infty)=n_+>0,\, \,\,\,{\tilq}(-\infty)=q_-,\,\,\,\, {\tilq}(+\infty)=q_+\end{equation} (we denote $\displaystyle \lim_{x\to \pm\infty} {f}(x)$ by  ${f}(\pm\infty)$ in short), provided the two end-states satisfy the Rankine-Hugoniot condition and the Lax entropy condition: \begin{align}\begin{aligned}\label{end-con_old} &\exists~\sigma\in\bbr \mbox{ such that }~\left\{ \begin{array}{ll}       -\sigma (n_+-n_-) - (n_+q_+ - n_-q_-) =0,\\       -\sigma (q_+ -q_-) -(n_+-n_-)=0, \end{array} \right. \\&\mbox{and either $n_->n_+$ and $q_-<q_+$ or  $n_-<n_+$ and $q_-<q_+$ holds.}        \end{aligned}\end{align} 
For notational convenience, we denote $\tilu(x-\sigma t)$ by $\hatu=(\hatn,\hat{q}):=\tilu(x-\sigma t)$ whenever there is no confusion about the wave $\tilu$ with its fixed boundary condition.\\

In short, for any $\nu>0$,  for any $n_->0$, for any $n_+>0$ with $n_+\neq n_-$ and for any $q_-\in\mathbb{R}$, there exists a   smooth monotone  traveling wave   $\hatu(t,x)=\tilu(x-\sigma t)$ of  \eqref{nq_nu}  
satisfying \eqref{bdry_cond}
 where 
 the constants $\sigma$ and
 $q_+$ are determined by   
\be\label{sigma_eq} \sigma:= \begin{cases}
 &\frac{-q_-+\sqrt{q_-^2+4n_+}}{2}>0\quad \mbox{ if }\quad n_->n_+>0 ,\\
&\frac{-q_--\sqrt{q_-^2+4n_+}}{2}<0 \quad \mbox{ if }\quad  0<n_-<n_+ 
\end{cases}\ee 
 and 
  \be\label{end-con} q_+:= q_-+\frac{(n_--n_+)}{\sigma}\ee

Our motivation of  this work is to answer  the question  how stable traveling waves are in the system.  The paper \cite{LiWa_siam} showed that waves  are stable   if 
the anti-derivative of a perturbation $(n-\tiln, q-\tilq) $ is sufficiently small in  the Sobolev space $H^2(\bbr)$. Note that  the initial perturbation should  have   the mean-zero condition: 
\ben
\exists x_0\in\bbr \mbox{ such that }\int_\bbr \bmat{ n_0(x)-\tiln(x-x_0)}\\ { q_0(x)-\tilq(x-x_0)}\emat  dx = \bmat{0}\\ {0}\emat. 
\een This restriction for the initial data is commonly assumed in studying stability of viscous shocks since     the work of \cite{Go} and \cite{KaMa}. The main novelty of this paper is to remove both the mean-zero condition  and the smallness condition of the initial perturbation. \\

In this paper, we  frequently use the following facts (\textit{e.g.} see \cite[Lemma 2.1]{ckkv2019}):\\
$$\tiln>0,\quad \tiln,\,\tilq,\,\frac{1}{\tiln}\in L^\infty(\bbr),\quad \mbox{and}\quad
\tiln',\tiln'',\tilq'\in L^1(\bbr)\cap L^\infty(\bbr).$$

\begin{subsection}{Global existence around waves and their contraction}\ \\

To state the contraction property, we need the following notion:\\
For $U_i=\begin{pmatrix} n_i   \\ q_i  \end{pmatrix}
$  with $n_i>0$ for $i=1,2$, we consider the relative entropy 
$$\eta(U_1 | U_2):= \frac{|q_1-q_2|^2}{2}+{\Pi}(n_1 | n_2),$$
 where
\be\label{def_rel_n} {\Pi}(n_1 | n_2):= {\Pi}(n_1)-{\Pi}(n_2)-\nabla{\Pi}(n_2)(n_1-n_2),\qquad \Pi(n):=n\log n -n.\ee
Since $\Pi(n)$ is strictly convex in $n$, its relative functional ${\Pi}(\cdot | \cdot)$ above is positive definite, and so is $\eta(\cdot | \cdot)$. That is,
  $\eta(U_1 | U_2)\geq 0$ for any $U_1$ and $U_2$, and 
$\eta(U_1 | U_2)= 0$ if and only if $U_1=U_2$.\\

We present our main result for  the fixed viscosity $\nu=1$ case: \begin{align}\label{nq} \begin{aligned}
\pa_t n -\pa_x(n q)& =   \pa_{xx} n,\\
\pa_t q -\pa_xn & = 0 \quad \mbox{ for } x\in\mathbb{R} \quad\mbox{and for }t>0,
\end{aligned} \end{align} 
assuming the case of $n_->n_+>0$. 
    Then, in Remark \ref{rem_scaling} and 
    \ref{rem_sigma_negative}, we illustrate that the main result still holds for any $\nu>0$ and/or for $n_+>n_->0$.

 For a given wave $\tiln$ and for a given constant    $\lambda>0$,   we define the weight function $a(\cdot)$ by 
\be\la{def_a}
a:=1+\frac{\lambda}{\vep}(n_--\tiln) 
\ee where $\epsilon:=(n_--n_+)>0$.
Then we have 
$
a(-\infty)=1, a(+\infty)=1+\lambda,\mbox{ and }a'(x)=\Big(-\frac{\lambda}{\vep}\Big)\tiln'(x)>0 \mbox{ for } x\in\bbr.
$  Here is the main result:

\begin{theorem}\label{main_thm}
For a given constant state $(n_-,q_-)\in\bbr^+\times\bbr$, 
there exist     constants  $\kappa\in(0,\min\{n_-/(15),1/8\})$ and $C>0$ 
such that the following is true:\\

For any $(n_+,q_+)\in\bbr^+\times\bbr$ satisfying \eqref{end-con} with $0<\epsilon:=(n_--n_+)<\kappa$, consider the traveling wave $\tilu:=\bmat\tiln \\ \tilq\emat$  of \eqref{nq}  with the boundary condition \eqref{bdry_cond} and  with the speed $\s$ from \eqref{sigma_eq}. Take  any constant $\lambda $ between $\frac{\epsilon}{\sqrt{\kappa}}$ and $\sqrt{\kappa}$. 
Let  $U_0(x):=\bmat n_0(x) \\ q_0(x) \emat $ satisfy 
\ben 
U_0-\tilu \in H^1(\mathbb{R}), \quad
0<\frac{1}{n_0}\in L^\infty(\bbr)
\een 
{\bf (i) Global existence :} 
Then 
there exists the unique global-in-time solution $U(t,x):=\bmat n(t,x)\\ q(t,x) \emat$   to \eqref{nq}      for    $U|_{t=0}=U_0$ such that
\begin{align*}
\begin{aligned}
&(n-\hatn, q-\hat q) \in\left(C([0,T];H^1(\bbr))\cap L^2(0,T;H^2(\bbr))\right)\times C([0,T];H^1(\bbr)),\\
&0<\frac{1}{n}\in L^\infty(0,T;L^\infty(\bbr))
\end{aligned}
\end{align*} 
 for any $T>0$.\\
 {\bf (ii) Contraction :}  Moreover, there exists an absolutely continuous shift function $X:[0,\infty)\rightarrow \mathbb{R}$ with 
$X\in W^{1,1}_{loc}$ and  
$X(0)=0$ such that   
\begin{align}
\begin{aligned}\la{ineq_contraction_up_to_infinity}
& \int_{-\infty}^{\infty} a(x-\s t) \eta\big(U(t,x-X(t))| \tilu(x-\s t)\big) dx \\
&\qquad + \sqrt{\kappa}\int_{0}^{t} \int_{-\infty}^{\infty} a(x-\s \tau) n\big(\tau,x-X(\tau)\big)\Big| \pa_x \Big(\log\frac{n(\tau,x-X(\tau))}{\tiln(x-\s \tau)}\Big)\Big|^2 dxd\tau\\
&\quad  \le \int_{-\infty}^{\infty} a(x) \eta\big(U_0(x)| \tilu(x)\big) dx,      \\
&\mbox{where $a$ is the monotone function defined by \eqref{def_a}}
\end{aligned}
\end{align} 
and
\be\begin{split}\la{est_shift_up_to_infinity}
&|\dot X(t)-\sigma|\le \frac{1}{\vep^2}\Big(f(t) + C\int_{-\infty}^{\infty} \eta(U_0| \tilu) dx +1  \Big) \quad \mbox{ for \textit{a.e.} }t\in[0,\infty)\\
&\mbox{where $f$ is  some positive function  satisfying}\quad\|f\|_{L^1(0,\infty)} \le C\frac{\lambda}{\eps}\int_{-\infty}^{\infty} \eta(U_0| \tilu) dx.
\end{split}\ee    
 \end{theorem}

 The proof is presented in Section \ref{sec_pf_main_thm}.

 
\begin{remark} \label{rem_sigma_negative} 
The result for $n_+>n_->0$ can be  obtained by the change of variables $x\mapsto -x$ with $\s\mapsto -\s$. 
Therefore, from now on, we always assume $n_->n_+>0$ and thus 
\[
0<\sigma=\frac{-q_- +\sqrt{q_-^2+4n_+}}{2}.
\]
\end{remark}

\begin{remark} \label{rem_scaling}
For general $\nu>0$ of \eqref{nq_nu}, we have the global existence and the contraction by the following scaling: \\
If $U^{\nu}$ and $\tilu^\nu$ are a solution and traveling wave to \eqref{nq_nu} for a fixed $\nu>0$ with initial data $U_0$, respectively, then  $U(t,x):=U^{\nu}(\nu t,\nu x)$ (resp. $\tilu(x):=\tilu^\nu(\nu x)$) is a solution (resp. traveling wave) to \eqref{nq}   (\textit{e.g.} also see \cite[Remark 1.5]{ckkv2019}).
 \end{remark}

\begin{remark}\label{rem_L_2}
For   $n_->0$, there exists a constant $C>0$ such that for any $n_1>0$ and for any  $n_2\in(n_-/2,n_-)$,
\be\label{Q_global_L2} \Pi(n_1|n_2) \leq C |n_1-n_2|^2\ee by \eqref{Q_loc_} and \eqref{upper_1/2} in Lemma \ref{lem_2.8} (or see \cite[Lemma 2.8]{ckkv2019}).
If we take $n_+\geq n_-/2$,
 it implies 
$n_-/2< \tiln< n_-$. Thus we have 
$$ \int_\bbr\eta(U(x) |\tilu(x))dx\leq C\|U-\tilu\|^2_{L^2(\bbr)}$$
   for any function $U$ with $U-\tilu\in L^2$. Therefore, the initial condition $U_0-\tilu\in H^1$ implies $ \int_{-\infty}^{\infty} \eta\big(U_0| \tilu\big)<\infty$.
   However, the reversed inequality is false because $\Pi\sim n_1\log n_1$ when $n_1$ is large (see \eqref{def_rel_n} and  \eqref{Q_global_} in Lemma \ref{lem_2.8}).
   \end{remark}

\begin{remark}
Since the weight function $a$ satisfies that $|a(x)-1|\leq\lambda<
\sqrt{\kappa}
<1/2$ for all $x\in\bbr$, the contraction estimate 
\eqref{ineq_contraction_up_to_infinity}
yields
\[
 \int_{-\infty}^{\infty} \eta\big(U(t,x-X(t))| \tilu(x-\s t)\big) dx   \le 4 \int_{-\infty}^{\infty} \eta\big(U_0(x)| \tilu(x)\big) dx.    
\]
\end{remark}

\vspace{1cm}

In the previous work \cite{ckkv2019},  it  was turned out that  both the smallness of the shock strength and the strict positivity of $n_-$ and $n_+$ in \eqref{bdry_cond} are technically important for our result even though the traveling waves exist even in the case of the large shock strength (or/and) $\min(n_-,n_+)=0$. In particular, as explained in \cite{LiWang12}, the case of $\min(n_-,n_+)=0$ is more relevant to  the original modeling.  The problem of the extension of our result 
seems   to be beyond reach of current known methods.
 With the mean-zero condition, the  stability for the case of $\min(n_-,n_+)=0$ case were shown 
 in a weighted Sobolev space
  in \cite{JinLiWa} and \cite{LiLiWa}. For planar waves on a cylinder, we refer to \cite{CCKL} and  \cite{CC2019}. \\

     For the Cauchy problem of \eqref{nq_nu}, we refer to
\cite{GXZZ, LPZ, MWZ}. For multi-dimentional cases, see \cite{LLZ} and references therein. 
 
 \end{subsection}

\subsection{Ideas of Proof}
In order to construct a global-in-time solution as a large $H^1$-perturbation of the traveling wave $\tilde U$, we may first find the usual relative entropy inequality for the system \eqref{nq}. 
For that, we need to observe the evolution of the relative entropy, based on the relative entropy method  \cite{Dafermos1,DiPerna}. More precisely, using the computations in the proof of \cite[Lemma 2.3]{ckkv2019} (or see \cite{Kang19,KV_arxiv,KV-unique19,Kang-V-1,Vasseur_Book}), we find that
\beq\label{st-eq}
\partial_t \eta(U|\tilde U) = -\partial_\xi \Big(G(U;\tilde U) + (\partial_\xi n) \log (n/\tilde n)\Big) -  \frac{|\partial_\xi n|^2}{n} + \frac{\partial_\xi n \tilde n'}{\tilde n} -\frac{n-\tilde n}{\tilde n} \tilde n''
+\frac{\tilde n'}{\tilde n} (n-\tilde n)(q-\tilde q) ,
\eeq
where $\xi:=x-\sigma t$, and $G(U;\tilde U)$ denotes the flux of the relative entropy.\\
If $\tilde n(\xi)$ were constant in $\xi$ like the case of $n_-=n_+$, then the above equality would become
\[
\partial_t \eta(U|\tilde U) = -\partial_\xi \Big(G(U;\tilde U) + (\partial_\xi n) \log (n/\tilde n)\Big)  
-  \frac{|\partial_\xi n|^2}{n},
\]
which gives the dissipation of the (total) relative entropy :
\be\label{simple_re}
\frac{d}{dt} \int_\bbr \eta(U|\hat U) dx +  \int_\bbr  \frac{|\partial_x n|^2}{n} dx \le 0 .
\ee
Note that the above inequality (in fact, contraction of the relative entropy) holds regardless of $q_- \neq q_+$ or $q_- = q_+$, i.e., discrepancy of the end states of $\hat q$.\\
However, we consider the traveling wave connecting two different states, that is, $\tilde n$ is not constant. Therefore, it is not obvious to get such a simple relative entropy functional inequality \eqref{simple_re} from \eqref{st-eq}.
In fact, it turns out in \cite{ckkv2019} that that is a far complicated issue. There, it was proven that the weighted relative entropy is dissipative (or contractive) up to a time-dependent shift $X(t)$ (see Proposition \ref{main_thm_ckkv}). 
Therefore, Proposition \ref{main_thm_ckkv} on the contraction property of the relative entropy will be importantly used in Proposition \ref{prop_unif} to extend the life span of a local-in-time solution for all time.\\
  

We sketch the proof.
 Recall that Proposition \ref{main_thm_ckkv} holds during $n>0$ \textit{i.e.} $1/n\in L^\infty$ (see the definition of the space \eqref{sp-T}). Thus, we first show a local existence theorem 
(Proposition  \ref{prop_lwp}) guaranteeing 
 that $n$ does not vanish up to a certain time interval $[0,T]$. Then we apply Proposition \ref{main_thm_ckkv} for the time interval in order to get the contraction of the weighted relative entropy functional \eqref{ineq_contraction} up to some shift $X(t)$ satisfying \eqref{est_shift}. In short, we have \begin{align}\label{in_short}
\begin{aligned}
& \frac{d}{dt} \int_{\bbr}  \hat{a}^{X} \eta\big(U| \hatu^{X}\big) dx+ \sqrt{\kappa}  \int_{\bbr} \hat{a}^{X} n \Big| \pa_x \Big(\log\frac{n }{\hatn^{X}}\Big)\Big|^2 dx  \leq 0,
\end{aligned}
\end{align} {where $\hat{a}(t,x)=a(x-\sigma t)$   with \eqref{def_a} and the superscript $X$ is defined by the translation in  $x-$variable by the given shift $X(t)$ as in \eqref{superX}.
 
After the process, it remains to solve  two main issues. First we  obtain finiteness 
(see \eqref{est_no_shift}) of the functional without a shift $X$ and without a weight $\hat a$: 
$$
\sup_{[0,T]}\int_{\bbr}   \eta\big(U| \hatu \big) dx   \leq C(T),
$$
thanks to boundedness of the  shift \eqref{est_shift}.
In this step, the estimate is little delicate  due to the Log structure of the relative entropy at infinity (see \eqref{def_rel_n} and \eqref{Q_global_}).   \\

Second, we obtain $q\in L^\infty$ by using the particular structure \eqref{w-eq} satisfied by $(n-\pa_x q)$. Here we take advantage of \eqref{w-ineq} from positivity of $n$. Since the dissipation term in \eqref{in_short} give the estimate of $\pa_x\sqrt{n}\in L^2$ (see \eqref{sqrt_n_diss}), we obtain $q\in L^\infty$ by decomposing each function into $L^1+L^\infty$.  Then the estimate $n, 1/n\in L^\infty$ follows from De Giorgi type Lemma \ref{lem_degiorgi}. By having $n,q\in L^\infty$,   the standard energy method gives all higher order estimates.\\

As a result, we get a priori bound in $H^1$-norm up to any arbitrarily large time, which  guarantee a $L^\infty$-bound of $1/n$ up to the life span of any solution due to De Giorgi type Lemma \ref{lem_degiorgi}. It implies no finite-time blow-up happens. In other words, there is a global-in-time solution.


 \subsection{A chemotaxis model describing tumour angiogenesis}\label{fromchemo}
 The system \eqref{nq_nu} can be derived from   the following  system of Keller-Segel type \cite{KSb}: 
 \begin{align}\label{KS} \begin{aligned}
\pa_t n - \nu\Del n& = - \na \cdot (n \chi(c) \na c ), \\
\pa_t c 
& = - c^m n  \quad \mbox{ for } \bold{x}\in\mathbb{R}^d \quad\mbox{and for }t>0. 
\end{aligned} \end{align}  

This system has been used to describe   chemotaxis  phenomena  including angiogenesis
that is
the formation of new blood vessels from pre-existing vessels. We may consider the formation   as  the mechanism for tumour progression and metastasis  (\textit{e.g.} see \cite{FoFrHu, FrTe,Le,Pe,  Rosen1, Sh}, and references therein).
In this interpretation,  we consider $n(\bold{x},t)>0$  the density of endothelial cells and $c(\bold{x},t)$  the concentration of  the protein known as the vascular endothelial  growth factor(VEGF) or just tumour angiogenesis factor(TAF). The given sensitivity function  $\chi(\cdot): \bbr^+ \to  \bbr^+ $ is usually assumed to be  decreasing to reflect that  the chemosensitivity becomes  lower as the  concentration of the chemical does higher.  The positive exponent
$m$ of the chemical concentration represent  the consumption rate of   $c$ (see the introduction in \cite{ckkv2019} for more details). \\

 For the Cauchy problem of \eqref{KS}, we see \cite{CPZ1, FrTe} and references therein.  We refer to the study on   traveling wave solutions of a Keller-Segel model in \cite{KSb} and many other works including \cite{Hor} (also see  the survey paper \cite{Wa_survey}).\\

To derive our system \eqref{nq_nu}, we just take    $\chi(c)= c^{-1}$ and $m =1$ and  $d=1$, 
  into \eqref{KS} to get
\begin{align*} 
 \begin{aligned}
\pa_t n - \nu \pa_{xx} n& = - \pa_x \left(n \frac{ \pa_x c}{c} \right), \\
\pa_t c  & = - cn. 
\end{aligned} \end{align*}
   Thanks to the 
    restriction   $m =1$, we can treat the singularity   in $c$ of the sensitivity by the Cole-Hopf transformation
\begin{equation*}
 q := - \pa_x [\ln c ]= -\frac{\pa_x c}{c}.
\end{equation*}   After the transform, we have 
\eqref{nq_nu} as in \cite{WaHi}.  \textit{cf)} For the case $m\neq1$, we refer to the recent work \cite{caliwa} and references therein.

\indent
 
    \indent

  \ \\

\section{Preliminaries}\label{sec_pre}
In this section, we present some lemmas that will be used throughout the paper.

\subsection{Useful inequalities}

We here present some useful inequalities on $\Pi(\cdot|\cdot)$, which were proved in \cite[Lemma 2.8]{ckkv2019}.

\begin{lemma} {\bf (\cite[Lemma 2.8]{ckkv2019})} \la{lem_2.8}
For given constants $\delta\in(0,\frac{1}{2}]$ and $n_->0$, there exist positive constants $C_1=C_1(n_-), C_2=C_2(n_-,\delta)$ and $C_3=C_3(n_-,\delta)$  such 
that the following inequalities hold:\\
1) For any $n_1>0$ and any $n_2>0$ with $\frac{n_-}{2}<n_2<n_-$,  
\be\la{Q_loc_} \frac{1}{{C_{1}}}|n_1-n_2|^2\leq  {\Pi}(n_1| n_2)\leq  {{C_{1}}}|n_1-n_2|^2 \quad \mbox{whenever }\, |\frac{n_1}{n_2}-1|\leq\delta,\ee 
\be\la{Q_global_} \frac{1}{{C_{2}}}(1+n_1\log^+ \frac{n_1}{n_2})\leq {\Pi}(n_1| n_2)\leq {{C_{2}}}(1+n_1\log^+ \frac{n_1}{n_2}) \quad \mbox{whenever }\, |\frac{n_1}{n_2}-1|\geq\delta,\ee 
\be\la{upper_1/2}\begin{split}
 & \frac{1}{{C_3}} |n_1-n_2|\leq {\Pi}(n_1|n_2) \le C_3|n_1-n_2|^2 \quad \mbox{whenever} \quad |\frac{n_1}{n_2}-1|\geq\delta,
 \end{split}\ee
 where $\log^+(y)$ is the positive part of $\log(y)$.\\
2) For any $n_1,n_2,m>0$ satisfying $m\le n_2 \le n_1$ or $n_1\le n_2 \le m$,
\beq\label{Pi-mono}
\Pi(n_1|m) \ge \Pi(n_2|m). 
\eeq
\end{lemma}

\subsection{De Giorgi type lemma}\label{subsec_lower}\ \\  
We here present the following technical lemma, which may not be optimal but is enough for our purpose. This   lemma might be classical, but we present its proof in Appendix \ref{app:De} for completeness. 
The proof is based on the De Giorgi method \cite{DeG}.
  \begin{lemma}\label{lem_degiorgi}
Let $T_0>0$ and $R>0$. Then there exists a constant $M=M(T_0,R)>0$ with the following property:\\

 Let $T\in(0,T_0]$ and let $p_1, p_2, p_3$ be functions such that
 \beq\label{p123}
 p_1, p_2, p_3\in L^\infty((0,{T} )\times \bbr),\quad p_2, \partial_x p_2, \partial_x p_3\in L^2(0,T;L^2(\mathbb{R})).
 \eeq
Let $m\in L^\infty((0,{T} )\times \bbr)\cap C([0,T]\times\bbr)$ be a non-negative 
 function such that
  \be\label{eq_de}\begin{cases}
  & \partial_x m,  \partial_{xx} m,  \partial_{t} m\in L^2(0,T;L^2(\mathbb{R})), 
  \\
&  \partial_t m - \partial_{xx}m+ p_1 \partial_x m + m \partial_x(p_2+p_3)\leq 0, \\
& m=m_1+m_2\mbox{ with }  m_1\in L^\infty(0,T;L^2(\mathbb{R}))\mbox{ and }  m_2\in L^\infty(0,T;L^\infty(\mathbb{R})).
  \end{cases}\ee
  Assume 
\begin{align}\label{assump_lem_de}
\begin{aligned}
&\|m|_{t=0}\|_{L^\infty (\bbr)}+\||p_1|+|p_2|+|p_3|+|m_2|\|_{L^\infty((0,{T} )\times \bbr)}
+\||p_2|+|\partial_x p_3|\|_{L^2((0,{T} )\times \bbr)}
\leq R.
\end{aligned}
\end{align}   
  Then 
  $$  \|m\|_{L^\infty((0,T)\times \bbr)}\leq M .$$
  \end{lemma} 
  \begin{remark}
  We do not ask any quantitative bound but only finiteness 
  for the norms of
  $$
 \partial_x m,  \partial_{xx} m,  \partial_{t} m, \pa_x p_2\in {L^2(0,T;L^2(\bbr))}, m_1\in L^\infty(0,T; L^2(\bbr))
  $$ to ensure that all computations  in the proof make sense. 
  \end{remark}

\subsection{A priori contraction estimate}\label{sec_ckkv}

As in \cite{ckkv2019}, we define  the   space 
\begin{align}
\begin{aligned}\label{sp-T}
\mathcal{X}_T := \{ \bmat n\\ q \emat\in L^\infty ((0,T)\times\bbr)^2~|~ n>0, ~  n^{-1}\in L^\infty((0,T)\times \bbr),~\partial_x n \in L^2((0,T)\times \bbr) \} 
\end{aligned}
\end{align} for each $T>0$.\\
  
The following proposition on the contraction property is the main result of \cite{ckkv2019}.
 
\begin{proposition}\label{main_thm_ckkv} \cite[Theorem 1.2]{ckkv2019}
For a given constant state $(n_-,q_-)\in\bbr^+\times\bbr$, 
there exist  constants  $\delta_0\in(0,1/2)$ and $\hat{C}>0$ such that the following is true:\\
For any $\eps,\lambda>0$ with $\eps\in(0,n_-)$ and $\delta_0^{-1}\eps<\lambda<\delta_0$, and for any $(n_+,q_+)\in\bbr^+\times\bbr$ satisfying \eqref{end-con} with $|n_--n_+|=\eps$, there exists  a smooth monotone function $a:\bbr\to\bbr^+$ with $\lim_{x\to\pm\infty} a(x)=1+a_{\pm}$ for some  constants $a_-, a_+$ with $|a_+-a_-|=\lambda$ such that the following holds:\\
Let $\tilu:=\bmat\tiln \\ \tilq\emat$ be a  traveling wave of \eqref{nq} with the boundary condition \eqref{bdry_cond} and  with the speed $\s$ from \eqref{sigma_eq}.
For a given $T>0$, let $U(t,x):=\bmat n(t,x)\\ q(t,x) \emat$ be a solution to \eqref{nq}  belonging to $\mathcal{X}_T$ with initial data $U_0(x):=\bmat n_0(x) \\ q_0(x) \emat $ satisfying \be
\label{initial_entropy} \int_{-\infty}^{\infty} \eta(U_0| \tilu) dx<\infty.\ee Then there exists an absolutely continuous shift function $X:[0,T]\rightarrow \mathbb{R}$ with 
$X\in W^{1,1}_{loc}$ and 
$X(0)=0$ such that   
\begin{align}
\begin{aligned}\la{ineq_contraction}
& \int_{-\infty}^{\infty} a(x-\s t) \eta\big(U(t,x-X(t))| \tilu(x-\s t)\big) dx \\
&\qquad + \delta_0\int_{0}^{t} \int_{-\infty}^{\infty} a(x-\s \tau) n\big(\tau,x-X(\tau)\big)\Big| \pa_x \Big(\log\frac{n(\tau,x-X(\tau))}{\tiln(x-\s \tau)}\Big)\Big|^2 dxd\tau\\
&\quad  \le \int_{-\infty}^{\infty} a(x) \eta\big(U_0(x)| \tilu(x)\big) dx,      
\end{aligned}
\end{align} 
and
\be\begin{split}\la{est_shift}
&|\dot X(t)-\sigma|\le \frac{1}{\vep^2}\Big(f(t) + \hat{C}\int_{-\infty}^{\infty} \eta(U_0| \tilu) dx +1  \Big) \quad \mbox{ for \textit{a.e.} }t\in[0,T]\\
&\mbox{where $f$ is  some positive function  satisfying}\quad\|f\|_{L^1(0,T)} \le \hat{C}\frac{\lambda}{\eps}\int_{-\infty}^{\infty} \eta(U_0| \tilu) dx.
\end{split}\ee 
 \end{proposition}

  \begin{remark}\label{rem_validity of diffusion term}
The diffusion term in 
\eqref{ineq_contraction} makes sense for
 solutions   $U$  of \eqref{nq_nu} in the class  $\mathcal{X}_T$. Indeed, 
we find 
\[
\pa_x \Big(\log\frac{n(t, x+Y(t))}{\tiln}\Big) \in L^2((0,T)\times\bbr)\] for any continuous and bounded function $Y:[0,T]\rightarrow\bbr$. It follows from $\partial_x n \in L^2((0,T)\times\bbr)$,   $n^{-1}\in L^\infty((0,T)\times\bbr)$,  $\tilde n\in L^\infty(\bbr)$, and $\tilde n' \in L^2(\bbr)$.
\end{remark}


  \begin{remark}\label{rem_bound_of_shift}
 The estimate \eqref{est_shift} implies
 $$
 |X(t)|\leq \breve{C} \cdot\Big(
 \int_\bbr \eta(U_0| \tilu) dx
 +1\Big)\cdot(t+1)$$ for any $t\in[0,T]$  where the constant $\breve{C}$ depends only on   the initial parameters  $n_-,q_-,\eps,$ and $\lambda$. In particular, the constant $\breve{C}$ is independent of $T$. 
\end{remark}

 \section{Proof of Theorem \ref{main_thm}}\label{sec_pf_main_thm}
In this section, we present the proof of Theorem \ref{main_thm}. 

 \subsection{Local existence in $H^1$}\label{subsec_lwp}
 We first present the local-in-time existence.
 \begin{proposition}\label{prop_lwp}
Let two given constant states $(n_-,q_-)\in\bbr^+\times\bbr$ and
 $(n_+,q_+)\in\bbr^+\times\bbr$ satisfy
 $n_-\neq n_+$ and \eqref{end-con}. 
 Consider the traveling wave $\tilde U=\bmat\tiln \\ \tilq\emat$  of \eqref{nq}  with the boundary condition \eqref{bdry_cond} and  with the speed $\s$ from \eqref{sigma_eq}.
 For any  $M_0>0$ and any $r_0>0$, there exists ${\hat{T}}>0$ such that   the following is true:\\
For any initial datum  $U_0 =\bmat n_0\\ q_0\emat$ satisfying
\be \label{local:ini}
\|U_0-\tilde U\|_{ H^1(\mathbb{R})}\leq M_0\quad\mbox{and}\quad \inf_\bbr n_0 \ge r_0,
\ee 
there exists the unique  solution $U =\bmat n\\ q\emat$ to \eqref{nq} on $[0,{\hat{T}}]$ with the initial datum $(n_0,q_0)$ such that    
\beq\label{know_2}
(n-\hatn, q-\hat q) \in\left(C([0,{\hat{T}}];H^1(\bbr))\cap L^2(0,{\hat{T}};H^2(\bbr))\right)\times C([0,{\hat{T}}];H^1(\bbr)),
\eeq
  \be\label{2M0}
\sup_{t\in[0,{\hat{T}}]} \|U(t)-\hatu(t)) \|_{ H^1(\mathbb{R})} \leq 2M_0 \quad\mbox{and}\quad\inf_{t\in[0,{\hat{T}}]} \inf_{x\in\bbr} n(x,t) \ge \frac{r_0}{2}.
\ee 
\end{proposition}
\begin{proof}
The proof for local existence of strong solutions to the 1D hyperbolic-parabolic system such as \eqref{nq} follows quite standard methods. For completeness, we present the proof   in Appendix \ref{app:local}.
\end{proof}

\subsection{Proposition \ref{prop_unif} : \textit{a priori} uniform estimates}\label{subsec_unif}
To get the global-in-time existence, we present the main proposition on \textit{a priori} uniform estimates:
\begin{proposition}\label{prop_unif} 
Under the same hypotheses as in Theorem \ref{main_thm}, if $U$ is a solution of \eqref{nq} on $[0,T_0)$ for some $T_0>0$ such that 
\begin{align}
\begin{aligned}\label{ass-sol}
&(n-\tiln, q-\tilq) \in\left(C([0,T] ;H^1(\bbr))\cap L^2(0,T ;H^2(\bbr))\right)\times C([0,T] ;H^1(\bbr)),\\
&\mbox{and}\quad 0<\frac{1}{n}\in L^\infty(0,T;L^\infty(\bbr)),\quad \forall T\in(0,T_0).
\end{aligned}
\end{align} 
Then there exists a constant $C(T_0)$ 
 such that 
\ben
\sup_{t\in[0,T_0)}\|U(t) -\hatu(t)\|_{  H^1(\bbr)}\leq C(T_0) \quad\mbox{and}\quad \sup_{t\in[0,T_0)} \|1/n\|_{L^\infty(\bbr)}\le C(T_0).
\een
\end{proposition} 

The proof of the main Proposition \ref{prop_unif} will be handled in Section \ref{sec_prop}. Based on this Proposition, we here complete the proof of Theorem \ref{main_thm}.\\

\subsection{Proof of Theorem \ref{main_thm}} 
For a given constant state $(n_-,q_-)\in\bbr^+\times\bbr,$ let us take the constants $\delta_0\in(0,1/2)$ and $\hat{C}>0$ from Proposition \ref{main_thm_ckkv}. Then, choose any constant $\kappa>0$ so that $\kappa<\min\{(\delta_0)^2/2,n_-/(15)\}$.  
Consider 
any $(n_+,q_+)\in\bbr^+\times\bbr$ satisfying \eqref{end-con} with $0<|n_--n_+|<\kappa$.\\
Let $\eps:=|n_--n_+|$ and take any $\lambda$ between 
$\frac{\epsilon}{\sqrt{\kappa}}$ and $\sqrt{\kappa}$.
 Note that these constants $\eps, \lambda>0$ satisfy the conditions 
  $\eps\in(0,n_-)$ and $\delta_0^{-1}\eps<\lambda<\delta_0$ 
in 
Proposition \ref{main_thm_ckkv}. Then, we take the constant 
$\breve{C}>0$ from Remark \ref{rem_bound_of_shift}. \\
Consider the traveling wave $\tilu:=\bmat\tiln \\ \tilq\emat$  of \eqref{nq}  with the boundary condition \eqref{bdry_cond} and  with the speed $\s$ from \eqref{sigma_eq}.
Let  $U_0(x):=\bmat n_0(x) \\ q_0(x) \emat $ satisfy 
\ben 
U_0-\tilu \in H^1(\mathbb{R}), \quad n_0>0\mbox{ on } \bbr\quad\mbox{and}\quad
\frac{1}{n_0}\in L^\infty(\bbr).
\een
We observe that Proposition \ref{prop_lwp} together with  Remark 
\ref{rem_L_2}  ensures the (local) existence of  a  solution $U$ 
  of \eqref{nq}   on $[0,\hat{T}]$ for some $\hat{T}>0$
for $U|_{t=0}=U_0$ such that
$$\int_{-\infty}^{\infty} \eta(U_0| \tilu) dx<\infty,$$
$$(n-\hatn, q-\hat q) \in\left(C([0,\hat{T}];H^1(\bbr))\cap L^2(0,\hat{T};H^2(\bbr))\right)\times C([0,\hat{T}];H^1(\bbr)),
$$
$$n>0\mbox{ on } [0,\hat{T}]\times \bbr \quad\mbox{and}\quad\frac{1}{n}\in L^\infty(0,\hat{T};L^\infty(\bbr)).$$

Now, in order to  extend the solution $U$  for all time,
\be\label{assump}\mbox{suppose that there is no global-in-time solution}.\ee Then there exists the finite maximal time interval $[0,T_0)$ for some  $T_0\in({\hat{T}},\infty)$ for the  existence  \textit{i.e.,} there   exists a solution $U$  on $[0,T_0)$ such that
 \begin{align}
\begin{aligned}\label{ass-contra}
&(n-\tiln, q-\tilq) \in\left(C([0,T] ;H^1(\bbr))\cap L^2(0,T ;H^2(\bbr))\right)\times C([0,T] ;H^1(\bbr)),\\
&\mbox{and}\quad 0<\frac{1}{n}\in L^\infty(0,T;L^\infty(\bbr)),\quad \forall T\in(0,T_0),
\end{aligned}
\end{align} 
but  
$$\mbox{either}\quad \sup_{t\in[0,T_0)}\|U(t) -\hatu(t)\|_{  H^1(\bbr)}=\infty\quad \mbox{or}
\quad \inf_{t\in[0,T_0)} \inf_{x\in\bbr} n(x,t) = 0\quad \mbox{holds}.$$ 
However, Proposition \ref{prop_unif} and \eqref{ass-contra} implies 
\[
\sup_{t\in[0,T_0)}\|U(t) -\hatu(t)\|_{  H^1(\bbr)}\leq C(T_0)\quad\mbox{and}\quad \sup_{t\in[0,T_0)} \|1/n\|_{L^\infty(\bbr)}\le C(T_0),
\]
where the constant $C(T_0)$ is independent of $T<T_0$. Therefore, 
$$\sup_{t\in[0,T_0)}\|U(t) -\hatu(t)\|_{  H^1(\bbr)}<\infty\quad \mbox{and}
\quad \inf_{t\in[0,T_0)} \inf_{x\in\bbr} n(x,t) >0,$$ which produces a contradiction to the assumption
\eqref{assump}.
Therefore, we have a global solution. 
 The proof of uniqueness follows the same standard energy method such as   
Step 5 in Appendix \ref{app:local}. It proves the part (i).\\ 

For the part (ii), we first notice that
the global solution 
$U$ belongs to the class $\mathcal{X}_{T}$ (see \eqref{sp-T}) for any $T>0$. Indeed, since $U -\hatu\in  L^\infty(0,T; H^1(\bbr))$ and $  \partial_x \hatu\in  L^\infty(0,{T}; L^2(\bbr))$, we have $\partial_x n\in L^2((0,T)\times\bbr)$, which implies $U\in \mathcal{X}_{T}$. Thus we 
 apply 
Proposition \ref{main_thm_ckkv} (or \cite[Theorem 1.2]{ckkv2019}) for any  arbitrarily large time interval. 
We recall how the shift is constructed in the proof of \cite[Theorem 1.2]{ckkv2019}, on which it is defined in a certain constructive way solving the given O.D.E. defined in \cite[(3.2)]{ckkv2019} uniquely (see the explanation in Section 3.1 and  Appendix A in \cite{ckkv2019}).  Since the right-hand side of (3.2) in \cite{ckkv2019} is well defined uniquely for any time, we can construct a shift $X:[0,\infty)\to\bbr$ with the desired estimates \eqref{ineq_contraction_up_to_infinity} and \eqref{est_shift_up_to_infinity}.\\


Therefore, it only remains to prove Proposition  \ref{prop_unif}.

\section{Proof of Proposition \ref{prop_unif}}\label{sec_prop}
First we note that for any $T\in(0,T_0)$, 
the local solution $U$  we are considering
belongs to the class $\mathcal{X}_{T}$ (see \eqref{sp-T}) thanks to \eqref{ass-sol}.
In this section, $C$ denotes a positive constant which may change from line to line, and depends on the initial data and $T_0$, but independent of $T\in(0,T_0)$. 

\subsection{Uniform bound of the relative entropy}
We will use Proposition \ref{main_thm_ckkv} to show that 
\beq\label{s_rel}
\sup_{t\in [0,T]}\int_\bbr\eta(U(t)|\hatu(t))dx\leq C .
\eeq
For simplicity, we here use the following notation:\\ for any function $f:\mathbb{R}_{\geq0}\times\mathbb{R}\rightarrow\mathbb{R}$ and any shift $X:[0,\infty)\rightarrow \mathbb{R}$, 
 \be\label{superX} f^{\pm X}(t,x):=f(t,x\pm X(t)).\ee
First of all, since Remark \ref{rem_L_2} together with $1/2\leq a \leq 3/2$ yields
\be\begin{split}
\int_\bbr a\eta(U_0|\tilu)dx\leq \int_\bbr \eta(U_0|\tilu)dx\leq C \int_\bbr |U_0-\tilu|^2dx
\leq C\|U_0-\tilu\|^2_{H^1(\bbr)},
\end{split}\ee 
Proposition \ref{main_thm_ckkv} and Remark \ref{rem_bound_of_shift} imply that there exists a function $X$ on $[0,T]$ such that 
\be\label{imp_con}
\begin{split}&\sup_{t\in [0,T]} \int_\bbr a^{-\sigma t}\eta([U(t)]^{-X(t)}|\tilu^{-\sigma t})dx\\
&\quad\quad
+  \int_{0}^{T} \int_{-\infty}^{\infty} a^{-\sigma \tau} [n(\tau)]^{-X(\tau)} \cdot\Big| \pa_x \Big(\log\frac{[n(\tau)]^{-X(\tau)} }{\tiln^{-\sigma \tau}}\Big)\Big|^2 dxd\tau
\leq C\end{split}
\ee
and $$\sup_{t\in [0,T]}|X(t)|\leq C.$$   

For any $t\in[0,T]$, we have
\be\begin{split}\label{initial_decomp}
\int_\bbr\eta(U(t)|\hatu(t))dx&=\int_\bbr\eta(U(t)|\tilu^{-\sigma t})dx
=\int_\bbr\Pi(n(t)|\tiln^{-\sigma t})dx+\frac{1}{2}\int_\bbr |q(t)-\tilq^{-\sigma t}|^2dx.
\end{split}\ee
For the second term in \eqref{initial_decomp}, we have
\be\begin{split}\label{q_trans}
\int_\bbr |q(t)-\tilq^{-\sigma t}|^2dx
&
\leq 2 \int_\bbr |q(t)-\tilq^{X(t)-\sigma t}|^2dx
+2 \int_\bbr | \tilq^{X(t)-\sigma t}-\tilq^{-\sigma t}|^2dx\\
&=2\int_\bbr |[q(t)]^{-X(t)}-\tilq^{ -\sigma t}|^2dx
+2\int_\bbr | \tilq^{X(t) }-\tilq|^2dx\\
&\leq C\int_\bbr a ^{-\sigma t} |[q(t)]^{-X(t)}-\tilq^{ -\sigma t}|^2dx
+2|q_+-q_-|\cdot \int_\bbr | \tilq^{X(t) }-\tilq| dx\\
&\leq C \sup_{t\in [0,T]} \int_\bbr a^{-\sigma t}\eta([U(t)]^{-X(t)}|\tilu^{-\sigma t})dx
+2|q_+-q_-|^2 |X(t)| \\
&\leq C({T_0}+1)\leq C .
\end{split}\ee 
For the first term in \eqref{initial_decomp}, we have
\ben\begin{split}
 &\int_\bbr\Pi(n(t)|\tiln^{-\sigma t})dx 
 \\
 &=  \int_{\{x\in\bbr\,|\,|\frac{n(t)}{\tiln^{X(t)-\sigma t}}-1|<\frac{1}{2}\}}\Pi(n(t)|\tiln^{-\sigma t})dx+\int_{\{x\in\bbr\,|\,|\frac{n(t)}{\tiln^{X(t)-\sigma t}}-1|\geq \frac{1}{2}\}}\Pi(n(t)|\tiln^{-\sigma t })dx =:I_1+I_2 .
\end{split}\een For $I_1$, we use \eqref{Q_global_L2} to have
\ben\begin{split}
I_1&\leq    C\int_{\{|\frac{n(t)}{\tiln^{X(t)-\sigma t}}-1|<\frac{1}{2}\}}|n(t)-\tiln^{-\sigma t}|^2dx .
\end{split}\een  Then, as   in \eqref{q_trans}, we get
\be\begin{split}\label{final_I_1}
I_1
&\leq   C\int_{\{|\frac{n(t)}{\tiln^{X(t)-\sigma t}}-1|<\frac{1}{2}\}}|n(t)-\tiln^{X(t)-\sigma t}|^2dx + C\int_{\{|\frac{n(t)}{\tiln^{X(t)-\sigma t}}-1|<\frac{1}{2}\}}|\tiln^{X(t)-\sigma t}-\tiln^{-\sigma t}|^2dx \\
&\leq   C\int_{\{|\frac{[n(t)]^{-X(t)}}{\tiln^{-\sigma t}}-1|<\frac{1}{2}\}}|[n(t)]^{-X(t)}-\tiln^{-\sigma t}|^2dx + C\int_{\bbr }|\tiln^{X(t)-\sigma t}-\tiln^{-\sigma t}|^2dx \\
&\leq   C\int_{\{|\frac{[n(t)]^{-X(t)}}{\tiln^{-\sigma t}}-1|<\frac{1}{2}\}}\Pi([n(t)]^{-X(t)}|\tiln^{-\sigma t})dx + C\int_{\bbr }|\tiln^{X(t)}-\tiln|^2dx \\
&\leq   C\int a^{-\sigma t}\Pi([n(t)]^{-X(t)}|\tiln^{-\sigma t})dx + |n_--n_+|^2\cdot|X(t)|\leq C({T_0}+1) \leq C ,
\end{split}\ee where we used \eqref{Q_loc_} for the third inequality.\\
For $I_2$, we recall $0<(n_--n_+)<\kappa<n_-/(15)<n_-/4$, and so $n_-<\frac{4}{3}n_+$. Since
$n_+<\tiln<n_-$, we find that for any $Y\in\bbr$,
$$ 
 \hatn^Y\leq   \frac{4}{3}\hatn .
$$
Thus,
$$\frac{n}{\hatn^Y}-1\geq \frac{1}{2}\Rightarrow  \frac{n}{\hatn}-1\geq \frac{1}{8},$$
and
$$\frac{n}{\hatn^Y}-1\leq -\frac{1}{2}\Rightarrow  \frac{n}{\hatn}-1\leq -\frac{1}{3},$$
which yield
\ben\begin{split}&
\{ |\frac{n(t)}{\tiln^{X(t)-\sigma t}}-1|\geq \frac{1}{2}\}=\{|\frac{n(t)}{[\hatn(t)]^{X(t)}}-1|\geq \frac{1}{2}\} 
  \subset
 \{ |\frac{n(t)}{\hatn(t)}-1|\geq \frac{1}{8}\} .
 \end{split}\een   
Thus we get 
\ben\begin{split}
 I_2&=
\int_{\{|\frac{n(t)}{[\hatn(t)]^{X(t)}}-1|\geq \frac{1}{2}\}}\Pi(n(t)|\hatn(t))dx \leq
\int_{\{|\frac{n(t)}{\hatn(t) }-1|\geq \frac{1}{8}\}}\Pi(n(t)|\hatn(t))dx .
 \end{split}\een We drop the $t$ index for simplicity. Then, by \eqref{Q_global_}, we get
 \ben\begin{split}
 I_2& \leq
\int_{\{|\frac{n}{\hatn }-1|\geq \frac{1}{8}\}}\Pi(n|\hatn)dx \leq C\int_{\{|\frac{n}{\hatn }-1|\geq \frac{1}{8}\}}(1+ n\log^+ \frac{n}{\hatn})dx ,\\
 \end{split}\een
Since the assumption $0<(n_--n_+)<\kappa<n_-/(15) $ implies    
\ben\label{1514}
 \hatn^Y\leq   \frac{15}{14}\hatn , 
\een
we have that for any $Y$, 
$$ \{|\frac{n}{\hatn }-1|\geq \frac{1}{8}\}\subset \{|\frac{n}{\hatn^{Y} }-1|\geq \frac{1}{20}\} .$$ 
Observe that for any point on  $\{|\frac{n}{\hatn }-1|\geq \frac{1}{8}\}$ and for any $Y\in\bbr$, we have
\be\label{outside_esti}(1+ n\log^+ \frac{n}{\hatn}) \leq C (1+ n \log^+ \frac{n }{\hatn^{Y}}).\ee Indeed, if $\frac{n}{\hatn}-1<-1/8$, then the estimate \eqref{outside_esti} is trivial due to $n<\hatn$. If  $\frac{n}{\hatn}-1>1/8$ \textit{i.e.} $n>\frac{9}{8}\hatn$, then we have $\frac{n}{\hatn}\leq \frac{15}{14}\cdot\frac{n}{\hatn^Y}$ and $n>\frac{21}{20}\hatn^Y>\hatn^Y $ from 
\eqref{1514},
 so we get 
 \ben\begin{split}&(1+ n\log^+ \frac{n}{\hatn})
 =(1+ n\log  \frac{n}{\hatn}) 
  \leq   (1+n\log\frac{15}{14}+ n \log  \frac{n }{\hatn^{Y}})\\
  &\quad \leq   (1+n_-\cdot \log\frac{15}{14}+ n \log  \frac{n }{\hatn^{Y}})
  \leq   C(1 + n \log  \frac{n }{\hatn^{Y}})=   C(1 + n \log^+  \frac{n }{\hatn^{Y}}).
  \end{split}\een
 Thus, by \eqref{Q_global_},  we get
 \be\begin{split}\label{final_I_2}
 I_2
&\leq C\int_{\{|\frac{n}{\hatn^{X} }-1|\geq \frac{1}{20}\}}(1+ n \log^+ \frac{n }{\hatn^{X}})dx \\
&\leq C
\int_{\{|\frac{n}{\hatn^{X} }-1|\geq \frac{1}{20}\}}\Pi(n |\hatn^{X})dx \leq  C
\int_{\bbr}\Pi(n|\hatn^{X})dx=   C
\int_{\bbr}\Pi(n^{-X}|\hatn)dx  \\
&\leq C
\int_{\bbr}a^{-\sigma  t}\Pi([n(t)]^{-X(t)}|\tiln^{-\sigma t})dx \leq C. \\
 \end{split}\ee
Thus from \eqref{q_trans}, \eqref{final_I_1}, and \eqref{final_I_2}, we have  \be\label{est_no_shift}\sup_{t\in [0,T]}\int_\bbr\eta(U(t)|\hatu(t))dx\leq C({T_0}+1)\leq C,\ee
which gives \eqref{s_rel}.\\

\subsection{Uniform bounds on $\|q-\hat{q}\|_{L^2}$ and $\|n-\hat{n}\|_{L^1+L^2}$}
We will use \eqref{s_rel} to show that 
\be\label{q2}
\|q-\hat{q}\|^2_{L^\infty(0,T;L^2(\bbr))}\leq C,
\ee
and there exists functions $m_1, m_2$ such that 
\be\label{m1m2}
n-\hat{n}=m_1+m_2,\quad \|m_1\|_{L^\infty(0,T;L^1(\bbr))}+\|m_2\|^2_{L^\infty(0,T;L^2(\bbr))}\leq C.
\ee
First of all, the definition of $\eta$ and \eqref{est_no_shift} implies that 
\[
\|q-\hat{q}\|^2_{L^\infty(0,T;L^2)}\leq C({T_0}+1)\leq C.
\]
We define 
\be\label{def_m1_m2}
m_1:=(n-\hatn)\mathbf{1}_{\{|\frac{n}{\hatn}-1|\geq \frac{1}{2}\}} \quad \mbox{ and }\quad
m_2:=(n-\hatn)\mathbf{1}_{\{|\frac{n}{\hatn}-1|< \frac{1}{2}\}}, 
\ee
which yields $n-\hat{n}=m_1+m_2$.\\
We use \eqref{upper_1/2} to have
 \be\begin{split} \label{m1_est} 
& \|m_1\|_{L^\infty(0,T;L^1(\bbr))}=\|(n-\hatn)\mathbf{1}_{\{|\frac{n}{\hatn}-1|\geq \frac{1}{2}\}}\|_{L^\infty(0,T;L^1(\bbr))}=\sup_{t\in[0,T]}\int_{\{|\frac{n}{\hatn}-1|\geq \frac{1}{2}\}} |n(t)-\hatn(t)|dx \\&\quad
\leq C
\sup_{t\in[0,T]}\int_{\{|\frac{n}{\hatn}-1|\geq \frac{1}{2}\}}\Pi(n(t)|\hatn(t))dx 
\leq C
\sup_{t\in[0,T]}\int_{\bbr}\Pi(n(t)|\hatn(t))dx \leq C({T_0}+1)\leq C.
\end{split}\ee  
Using \eqref{Q_loc_}, we have
 \be\begin{split}\label{m2_est}
&\|m_2\|^2_{L^\infty(0,T;L^2(\bbr))}=\|(n-\hatn)\mathbf{1}_{\{|\frac{n}{\hatn}-1|< \frac{1}{2}\}} \|^2_{L^\infty(0,T;L^2(\bbr))} 
=\sup_{t\in[0,T]}\int_{\{|\frac{n}{\hatn}-1|< \frac{1}{2}\}} |n(t)-\hatn(t)|^2dx \\&\quad
\leq C
\sup_{t\in[0,T]}\int_{\{|\frac{n}{\hatn}-1|<\frac{1}{2}\}}\Pi(n(t)|\hatn(t))dx 
\leq C
\sup_{t\in[0,T]}\int_{\bbr}\Pi(n(t)|\hatn(t))dx \leq C({T_0}+1)\leq C .
\end{split}\ee
Therefore, we have \eqref{m1m2}.

\subsection{Uniform bound on $\|\partial_x\sqrt{n}\|_{L^2}$}
We will use \eqref{m1m2} and \eqref{imp_con} to get that
\beq\label{dsn2}
\int_{0}^{T} \int_{-\infty}^{\infty} |\pa_x\sqrt{n}|^2 dxd\tau \leq C.
\eeq
First, we find from \eqref{imp_con} that  
\begin{align*}
\begin{aligned}
 \int_{0}^{T} \int_{-\infty}^{\infty} a^{X(\tau)-\sigma \tau} n \Big| \pa_x \Big(\log\frac{n }{\tiln^{X(\tau)-\sigma \tau}}\Big)\Big|^2 dxd\tau  \leq C .
\end{aligned}
\end{align*}
Observe that for any $Y\in\bbr$,
 \ben\begin{split}
n\Big|\partial_x\Big(\log\frac{n}{\tiln^Y}\Big) \Big|^2 &=
\frac{1}{(\tiln^Y)^2}\cdot\frac{|(\partial_xn)\tiln^Y-n\partial_x\tiln^Y|^2}{n}
= \frac{4}{\tiln^Y}\cdot
\frac{1}{4(\tiln^Y)^3}\cdot\frac{|(\partial_xn)\tiln^Y-n\partial_x\tiln^Y|^2}{n}\\
&= \frac{4}{\tiln^Y}\cdot
 \Big| 
\partial_x\sqrt{\frac{n}{\tiln^Y}} 
 \Big|^2.
\end{split}\een 
Then, using the fact that $a$ and $\tiln$ are bounded from below and above by a positive constant, we have
\beq\label{int-diff}
\int_{0}^{T} \int_{-\infty}^{\infty} \left| \pa_x \sqrt{\frac{n }{\tiln^{X(\tau)-\sigma \tau}} }\right|^2 dxd\tau \le C.
\eeq
Note, for any $Y\in \bbr$,
\begin{align*}
\begin{aligned}
 \left| \pa_x \sqrt{\frac{n }{\tiln^Y}}\right| &=\left| \frac{((\pa_x n) \tiln^Y- n (\tiln^Y)')/ (\tiln^Y)^2}{ 2 \sqrt{n/\tiln^Y1}} \right| \ge C^{-1} \left| \frac{(\pa_x n) \tiln^Y- n(\tiln^Y)'}{ \sqrt{n}} \right| \\
&\ge C^{-1}(2\tiln^Y|\pa_x\sqrt{n}|-|(\tiln^Y)'|\sqrt{n}) ,
\end{aligned}
\end{align*}
and thus,
\[
|\pa_x\sqrt{n}| \le C  \left| \pa_x \sqrt{\frac{n }{\tiln^Y}}\right| +  C|(\tiln^Y)'|\sqrt{n}.
\]
Thus we have
\begin{align*}
\begin{aligned}
&\int_{0}^{T} \int_{-\infty}^{\infty} |\pa_x\sqrt{n}|^2 dxd\tau \\
& \le C \int_{0}^{T} \int_{-\infty}^{\infty} \left| \pa_x \sqrt{\frac{n }{\tiln^{X(\tau)-\sigma \tau}} }\right|^2 dxd\tau
+ C \underbrace{\int_{0}^{T} \int_{-\infty}^{\infty} |(\tiln')^{X(\tau)-\sigma \tau}|^2 n ~ dxd\tau }_{=:J}.
\end{aligned}
\end{align*} 
To control $J$, using
\beq\label{n-decom}
|n|\leq|n-\hat n|+|\hat n|\leq |m_1| +|m_2|+|\hat n| ,
\eeq
and $\tiln' \in L^\infty(\bbr)\cap L^2(\bbr)$, together with \eqref{m1m2}, we have
\[
J\le C\cdot {T_0}\cdot
 \Big( \|m_1\|_{L^\infty(0,T;L^1(\bbr))}  +\|m_2\|_{L^\infty(0,T;L^2(\bbr))} +1\Big) \le C {T_0} ({T_0}+1).
\]
This and \eqref{int-diff} yields
\be\label{sqrt_n_diss}
\int_{0}^{T} \int_{-\infty}^{\infty} |\pa_x\sqrt{n}|^2 dxd\tau \le C({T_0}+1)^2\leq C.
\ee

\subsection{Uniform bound on $\|q\|_{L^\infty}$ }
In order to get the uniform bounds for $\|n\|_{L^\infty(0,T;L^\infty(\bbr))}$ and $\|1/n\|_{L^\infty(0,T;L^\infty(\bbr))}$, we may first get $\|q\|_{L^\infty(0,T;L^\infty(\bbr))}\le C$ and then apply Lemma \ref{lem_degiorgi}.
So we will here show 
\be\label{q-est_infty}
\|q\|_{L^\infty(0,T;L^\infty(\bbr))} \leq C .
\ee
For that, we first use \eqref{n-decom} to find that for any $x\in\bbr$ and $t\in[0,T]$,
\begin{align*}
\begin{aligned}
\|n(t)\|_{L^1([x-1,x+1])} & \le \|m_1\|_{L^\infty(0,T;L^1(\bbr))} + \|m_2\|_{L^\infty(0,T;L^1([x-1,x+1]))} + 2\|\tiln\|_{L^\infty(\bbr)} \\
&\le \|m_1\|_{L^\infty(0,T;L^1(\bbr))} +\sqrt{2} \|m_2\|_{L^\infty(0,T;L^2(\bbr))} + 2\|\tiln\|_{L^\infty(\bbr)} .
\end{aligned}
\end{align*} 
So we have 
\beq\label{int-n}
\sup_{t\in[0,T]}\sup_{x\in\bbr}\|n(t)\|_{L^1([x-1,x+1])}\leq C({T_0}+1)\leq C.
\eeq
Since $$n(t,x)=n(t,y)+\int_y^x(\partial_xn)(t,z)dz$$ and 
$$(\partial_xn)=2\sqrt{n}\partial_x\sqrt{n},$$
we have
\[
n(t,x)=\frac{1}{2}\int_{x-1}^{x+1} n(t,y) dy + \int_{x-1}^{x+1} \int_y^x \sqrt{n} \partial_x\sqrt{n} dzdy.
\]
Then, we use \eqref{dsn2} and \eqref{int-n} to have
\begin{align*}
\begin{aligned}
n(t,x) &\le \frac{1}{2}\int_{x-1}^{x+1} n(t,y) dy + \int_{x-1}^{x+1} \int_{x-1}^{x+1} \sqrt{n}|\partial_x\sqrt{n}| dzdy \\
&\le  \frac{1}{2} \|n(t)\|_{L^1([x-1,x+1])} + 2\sqrt{\int_{x-1}^{x+1} |n| dz} \sqrt{\int_{x-1}^{x+1} |\partial_x\sqrt{n}|^2 dz} \\
&\le  \frac{3}{2}\sup_{x\in\bbr} \|n(t)\|_{L^1([x-1,x+1])} + \|\partial_x\sqrt{n(t)}\|_{L^2(\bbr)}^2,
\end{aligned}
\end{align*} 
and thus,
\begin{align}
\begin{aligned}  \label{int-nest} 
\|n\|_{L^1(0,T;L^\infty(\bbr))} & \le \frac{3}{2}{T_0} \sup_{t\in[0,T]}\sup_{x\in\bbr}\|n(t)\|_{L^1([x-1,x+1])} +\|\partial_x\sqrt{n}\|_{L^2(0,T;L^2(\bbr))}^2 \\
&\le C({T_0}+1)^2\leq C .
\end{aligned}
\end{align} 

We now introduce, to show \eqref{q-est_infty}, a new variable
\[
w:=n-\partial_x q .
\]
Then, it follows from \eqref{nq} that
\be\label{w-eq}
\partial_tw+nw= n^2+q\partial_x n.
\ee
Since $n>0$, we have 
\beq\label{w-ineq} \partial_t|w|\leq n^2+|q\partial_x n|.\eeq
To estimate $n^2$, 
 we observe  
\[
n^2=n(n-\hat{n}+\hat{n})=n(m_1+m_2+\hat{n})=
\underbrace{nm_1 }_{=:k_1}+\underbrace{n(m_{2}+\hatn)}_{=:k_2}.
\] 
Since $|m_2|=|\hatn(\frac{n}{\hatn}-1)\mathbf{1}_{\{|\frac{n}{\hatn}-1|< \frac{1}{2}\}}| \leq  \frac{n_-}{2}\leq C$, we have \ \\
$$\||m_{2}|+\hatn\|_{L^\infty(0,T;L^\infty(\bbr) )}\leq C.$$ 
By \eqref{int-nest} and \eqref{m1m2}, 
we have
$n^2=k_1+k_2$ with
\beq\label{n2}\|k_1\|_{L^1(0,T; L^1(\bbr))}+\|k_2\|_{L^1(0,T; L^\infty(\bbr))}\leq C({T_0}+1)^3\leq C.\eeq
To estimate $|q \partial_x n|$, we first observe that
since $\partial_x n=2\sqrt{n}\partial_x\sqrt{n}$ with    
$$\|\pa_x\sqrt n\|_{L^2(0,T;L^2(\bbr))}\leq C({T_0}+1)\leq C$$ 
by \eqref{dsn2} and $$\|\sqrt{n}\|_{ L^2(0,T;L^\infty(\bbr))}\leq C({T_0}+1)\leq C$$ by \eqref{int-nest}, 
  we have
\[
\|\partial_x n \|_{L^1(0,T;L^2(\bbr))} \le C({T_0}+1)^2\leq C.
\] It implies
$$ 
\|(q-\hat q)\cdot\pa_x n\|_{L^1(0,T; L^1(\bbr) ) }  
\le C({T_0}+1)^{5/2}\leq C .
$$
Note  $\sqrt{n}\leq \sqrt{|m_1|}+\sqrt{|m_2+\hatn|}
$ from $n=m_1+m_2+\hatn$ with
$$\|\sqrt{|m_1|}\|_{ L^\infty(0,T;L^2(\bbr))}
\leq \|{m_1}\|^{1/2}_{ L^\infty(0,T;L^1(\bbr))} \leq C({T_0}+1)^{1/2}\leq C$$ and
$\|\sqrt{|m_2+\hatn|}\|_{ L^\infty(0,T;L^\infty(\bbr))}\leq C$.\\
Thus we get
$|\partial_x n|=2|\sqrt{n}|\cdot|\partial_x\sqrt{n}|\leq C
\Big(\underbrace{|\sqrt{|m_1|}|\cdot|\partial_x\sqrt{n}|}_{=:h_1}
+\underbrace{(|\sqrt{|m_2+\hatn|}|\cdot|\partial_x\sqrt{n}|}_{=:h_2}
\Big)
$ with
$$\|
h_1
\|_{ L^1(0,T;L^1(\bbr))}
   \leq \sqrt{{T_0}}\cdot\|
h_1
   \|_{ L^2(0,T;L^1(\bbr))}
   \leq C({T_0}+1)^{2}\leq C$$  and
   $$\|
h_2
   \|_{ L^2(0,T;L^2(\bbr))}
      \leq C({T_0}+1)\leq C $$
We put $h_2=\underbrace{h_2\mathbf{1}_{\{|h_2|>1\}}}_{=:h_{2,1}}+\underbrace{h_2\mathbf{1}_{\{|h_2|\leq1\}}}_{=:h_{2,2}}$, then we get
   $$\|
h_{2,1} 
   \|_{ L^1(0,T;L^1(\bbr))}\leq \|
h_{2,1} 
   \|^2_{ L^2(0,T;L^2(\bbr))}
      \leq C({T_0}+1)^2\leq C $$ and $$\|
h_{2,2} 
   \|_{ L^1(0,T;L^\infty(\bbr))}
      \leq {T_0}\cdot  \|
h_{2,2} 
   \|_{ L^\infty(0,T;L^\infty(\bbr))}
      \leq C\cdot {T_0}. \leq C.$$
    Thus we have
    $|\hat q \pa_x n|\leq C |\hat q|\cdot (h_1 +h_{2,1}+h_{2,2})=
\underbrace{   C |\hat q|\cdot (h_1 +h_{2,1})}_{=:l_1}  +\underbrace{C|\hat q|\cdot  h_{2,2}}_{=:l_2} $ with
   $$\|
l_{1} 
   \|_{ L^1(0,T;L^1(\bbr))}\leq  C({T_0}+1)^2 \leq C$$ and
   $$\|
l_{2} 
   \|_{ L^1(0,T;L^\infty (\bbr))}\leq  C\cdot {T_0}\leq C $$ and

In sum, we have
$|q\partial_x n|\leq\underbrace{|(q-\hat q)\cdot\pa_x n|}_{=:l_0}+l_1+l_2$ with
 \beq\label{qdn}
\|l_0+l_1 \|_{L^1(0,T; L^1(\bbr) ) } +
\|l_2 \|_{L^1(0,T;  L^\infty(\bbr)) } 
\le C({T_0}+1)^{5/2}\leq C .
\eeq

Therefore, it follows from \eqref{w-ineq}, \eqref{n2} and \eqref{qdn} that $$\pa_t |w|\leq \underbrace{(k_1+l_0+l_1)}_{=:w_1}+\underbrace{(k_2+l_2)}_{=:w_2}$$ with
 \ben
\|w_1\|_{L^1(0,T; L^1(\bbr) ) } +
\|w_2 \|_{L^1(0,T;  L^\infty(\bbr)) } 
\le C({T_0}+1)^{3} \leq C.
\een
Moreover, since

$w_0=n_0-\pa_x q_0=
\underbrace{-\pa_x(q_0-\tilq)\mathbf{1}_{\{|\pa_x(q_0-\tilq)|>1\}}}_{=:j_1}+\underbrace{\Big(n_0-\pa_x(q_0-\tilq)\mathbf{1}_{\{|\pa_x(q_0-\tilq)|\leq1\}}-\pa_x\tilq\Big)}_{=:j_2}$
with
$$\|j_1\|_{L^1(\bbr)}\leq \|\pa_x(q_0-\tilq)\|^2_{L^2(\bbr)} 
\leq \|U_0-\tilu\|^2_{H^1(\bbr)} \leq C$$ and
$$\|j_2\|_{L^\infty(\bbr)}\leq \|n_0\|_{L^\infty(\bbr)}+ C
\leq \|n_0-\tiln\|_{L^\infty(\bbr)}+ C\leq \|U_0-\tilu\|_{H^1(\bbr)}+ C\leq C.$$ 

Therefore,
we have 
$|w|\leq i_1+i_2$ with
 \ben
\|i_1\|_{L^\infty(0,T; L^1(\bbr) ) } +
\|i_2 \|_{L^\infty(0,T;  L^\infty(\bbr)) } 
\le C({T_0}+1)^{3}\leq C .
\een
Indeed, for $x\in\bbr$ and for $t\in[0,T]$, we have
$$|w(t,x)|=|w_0(x)|+\int_0^t(\pa_t|w|)(s,x)ds
\leq \underbrace{ |j_1(x)|+\int_0^t w_1(s,x) ds}_{=:\alpha_1(t,x)}
+ \underbrace{\Big(|j_2(x)|+\int_0^t w_2(s,x) ds}_{=:\alpha_2(t,x)}
\Big)$$ with
 \ben
\|\alpha_1\|_{L^\infty(0,T; L^1(\bbr) ) } 
\le \|j_1\|_{L^1(\bbr ) } +\|w_1\|_{L^1(0,T; L^1(\bbr) ) } \le C({T_0}+1)^{3} \leq C
\een and 
 \ben
\|\alpha_2\|_{L^\infty(0,T; L^\infty(\bbr) ) } 
\le \|j_2\|_{L^\infty(\bbr ) } +\|w_2\|_{L^1(0,T; L^\infty(\bbr) ) } \le C({T_0}+1)^{3} \leq C.
\een
This implies 
\[
|\pa_x q| =| n -w| \leq |n-\hatn|+|\hatn|+|w|=\underbrace{ |m_1| +\alpha_1}_{=:g_1}+\underbrace{\Big(|m_2|+\hat n+\alpha_2 \Big)}_{=:g_2} 
\]
with
 \beq\label{dq-est}
\|g_1\|_{L^\infty(0,T; L^1(\bbr) ) } +
\|g_2 \|_{L^\infty(0,T;  L^\infty(\bbr)) } 
\le C({T_0}+1)^{3}\leq C.
\eeq
Note from \eqref{q2} that $q=q-\hat q +\hat q=\underbrace{(q-\hat q)\mathbf{1}_{\{|q-\hat  q |>1\}}}_{=:f_1}+\underbrace{(q-\hat q)\mathbf{1}_{\{|q-\hat  q |\leq 1\}}+\hat q}_{=:f_2}$
and
\beq\label{q-est_}
\|f_1\|_{L^\infty(0,T; L^1(\bbr) ) } +
\|f_2 \|_{L^\infty(0,T;  L^\infty(\bbr)) } 
\le C({T_0}+1)\leq C.
\eeq
Therefore, using Lemma \ref{lem:infty} below, together with \eqref{q-est_} and \eqref{dq-est}, we have
\[
\|q \|_{L^\infty(0,T; L^\infty(\bbr))} \le C({T_0}+1)^3\leq C .
\]
\begin{lemma}\label{lem:infty}
Let $f$  be any function on $\bbr$ such that $f=f_1+f_2$ and $  |f'|\leq g_1+g_2$ with $f_1,g_1\in L^1(\bbr)$ and $f_2,g_2\in L^\infty(\bbr)$.\\
Then,
 $f\in L^\infty(\bbr)$ with 
\[
\|f\|_{L^\infty(\bbr)}\le 2\Big(
 \|f_1\|_{L^1(\bbr)} +  \|f_2\|_{L^\infty(\bbr)} +\|g_1\|_{L^1(\bbr)} +  \|g_2\|_{L^\infty(\bbr)}  
\Big).
\]
\end{lemma}
\begin{proof}
Since
\[
f(x)=f(y)+\int_y^x f' (z)dz ,
\] for any $x,y\in\bbr$,
we have, by taking $\frac{1}{2}\int_{x-1}^{x+1}dx$,
\begin{align*}
\begin{aligned}
|f(x)| &\le \frac{1}{2} \int_{x-1}^{x+1}( |f_1(y)|+|f_2(y)| ) dy +\frac{1}{2} \int_{x-1}^{x+1} \int_{x-1}^{x+1} (|g_1(z)|+|g_2 (z)|) dz dy \\
&\le  \frac{1}{2}  \|f_1\|_{L^1(\bbr)} +  \|f_2\|_{L^\infty(\bbr)} +\|g_1\|_{L^1(\bbr)} + 2\|g_2\|_{L^\infty(\bbr)}.
\end{aligned}
\end{align*} for any $x\in\bbr$.
\end{proof}

\subsection{Uniform bounds on $\|n\|_{L^\infty}$ and $\|1/n\|_{L^\infty}$ }
We now use Lemma \ref{lem_degiorgi} (De Giorgi type lemma) to get uniform bounds on $\|n\|_{L^\infty(0,T;L^\infty(\bbr)}$ and $\|1/n\|_{L^\infty(0,T;L^\infty(\bbr))}$. 
First, to control $\|n\|_{L^\infty(0,T;L^\infty(\bbr))}$, we set
\be\label{eq_put_n}
 m=n,\quad m_1=n-\hat{n},\quad m_2=\hat{n},\quad p_1=-q, \quad p_2=-(q-\hat{q}),\quad\mbox{and}\quad p_3=-\hat{q}.\ee
Since 
$$ \partial_t n - \partial_{xx}n-q \partial_x n - n \partial_x q= 0, $$
the above quantities in \eqref{eq_put_n} satisfy the assumption of Lemma \ref{lem_degiorgi}. More precisely, we use  \eqref{q-est_infty} and \eqref{q2} to estimate
\begin{align}
\begin{aligned}
&\|m|_{t=0}\|_{L^\infty (\bbr)}+\||p_1|+|p_2|+|p_3|+|m_2|\|_{L^\infty((0,{T} )\times \bbr)} 
+\||p_2|+|\partial_x p_3|\|_{L^2((0,{T} )\times \bbr)}\\
&=\|n_0\|_{L^\infty (\bbr)}+\||q|+|q-\hat{q}|+|\hat{q}|+|\hat{n}|\|_{L^\infty((0,{T} )\times \bbr)} 
+\||q-\hat{q}|+|\partial_x \hat{q}|\|_{L^2((0,{T} )\times \bbr)}\\
&\leq \|n_0\|_{L^\infty (\bbr)}+
2\|{q}\|_{L^\infty((0,{T} )\times \bbr)} 
+2\||{\hat{q}}|+|{\hat{n}}|\|_{L^\infty((0,{T} )\times \bbr)} 
+\sqrt{T}\||{q-\hat{q}}|+|{\pa_x\hat{q}}|\|_{L^\infty(0,{T};L^2(\bbr))} \\
&\leq C(T_0+1)\leq C.
\end{aligned}
\end{align}   
Since the above constant $C$ does not depend on $T$, by Lemma \ref{lem_degiorgi}, we obtain
 \be\label{n-est_infty}\|n\|_{L^\infty(0,T; L^\infty(\bbr))}\leq C_{T_0}\leq C.\ee
Similarly, we can obtain
  \be\label{est_1/n}\|\frac{1}{n}\|_{L^\infty(0,T; L^\infty(\bbr))}\leq C_{T_0}\leq C.\ee 
Indeed, in order to apply Lemma \ref{lem_degiorgi}, let
\be\label{eq_put_1/n}
m=1/n,\quad m_1=\frac{1}{n}-\frac{1}{\hat{n}}=\frac{\hat{n}-n}{n\hat{n}},\quad m_2=\frac{1}{\hat{n}},\quad p_1=-q, \quad p_2=(q-\hat{q}),\quad\mbox{and}\quad p_3=\hat{q} .\ee
Notice that it follows from \eqref{nq} and \eqref{ass-sol}
 that
  $$ \partial_t \Big(\frac{1}{n}\Big) - \partial_{xx}\Big(\frac{1}{n}\Big)+(-q) \partial_x \Big(\frac{1}{n}\Big) + \Big(\frac{1}{n}\Big)\partial_x q
  = -\frac{2(\partial_x n)^2}{n^3}\leq  0,\quad\mbox{for a.e. } t\in [0,  T],
  $$ 
where
$\frac{(\pa_x n)^2}{n^3} \in L_t^2L_x^2$ by the interpolation
$\pa_x n\in L^\infty_t L^2_x\cap L^2_tL^\infty_x\subset L^4_tL^4_x$.
Thus,
\eqref{ass-sol} implies that the quantities of \eqref{eq_put_1/n} satisfy \eqref{p123} and \eqref{eq_de}  on $[0,\tilde T]$.\\
Furthermore, the quantities of \eqref{eq_put_1/n} satisfy \eqref{assump_lem_de} as follows: 
  \begin{align*}
\begin{aligned}
&\|m|_{t=0}\|_{L^\infty (\bbr)}+\||p_1|+|p_2|+|p_3|+|m_2|\|_{L^\infty((0,{T} )\times \bbr)} 
+\||p_2|+|\partial_x p_3|\|_{L^2((0,{T} )\times \bbr)}\\
&=\|1/n_0\|_{L^\infty (\bbr)}+\||q|+|q-\hat{q}|+|\hat{q}|+|1/\hat{n}|\|_{L^\infty((0,{T} )\times \bbr)} 
+\||q-\hat{q}|+|\partial_x \hat{q}|\|_{L^2((0,{T} )\times \bbr)}\\
&\leq C(T_0+1)\leq C.
\end{aligned}
\end{align*}  Thus \eqref{est_1/n} follows from Lemma \ref{lem_degiorgi}.

\subsection{Uniform bound on $\|n-\hat{n}\|_{L^2}$ }
We first recall from  \eqref{m1_est}, \eqref{m2_est}, \eqref{n-est_infty} that $m_1=(n-\hatn)-m_2$, and 
 $$\|m_1\|_{L^\infty(0,T; L^1(\bbr))}\leq  
 C({T_0}+1)\leq C,$$
 and 
 $$\|(n-\hatn)-m_2\|_{L^\infty(0,T; L^\infty(\bbr))}\leq 
\|n\|_{L^\infty(0,T; L^\infty(\bbr))}+C\leq C.$$
 Since
\begin{align*}
\begin{aligned}
\int_\bbr|m_1(t)|^2dx=\int_\bbr|m_1(t)|\cdot|
(n-\hatn)-m_2
|dx\leq  C_{T_0}({T_0}+1)\leq C
\end{aligned}
\end{align*} for any $t\in[0,T]$, we get
$$\|m_1\|^2_{L^\infty(0,T; L^2(\bbr))}\leq  
 C_{T_0}({T_0}+1)\leq C.$$
 Thus we have
 $$\|n-\hatn\|_{L^\infty(0,T; L^2(\bbr))}\leq  \|m_1\|_{L^\infty(0,T; L^2(\bbr))}+
\|m_2\|_{L^\infty(0,T; L^2(\bbr))} 
\leq   C_{T_0}({T_0}+1)^{1/2}\leq C$$
by \eqref{m2_est}.\\

\subsection{Uniform bounds on $\|\partial_x n\|_{L^2}$, $\|\partial_x q\|_{L^2}$ and $\|\partial_{xx} n\|_{L^2}$ } 

From the system \eqref{nq}, we do the energy method to obtain
\begin{align*}
\begin{aligned}
&\frac{d}{dt}\Big(\int_\bbr\frac{|\pa_xn|^2}{2}dx
+\int_\bbr\frac{|\pa_xq|^2}{2}dx\Big)+\int_\bbr|\pa_{xx}n|^2dx\\
&\quad\quad=\int_\bbr\Big(
(\pa_x n)(\pa_{xx} n)q+2(\pa_x n)^2\pa_x q-\pa_x[n (\pa_x n)](\pa_{x} q)+(\pa_x q)(\pa_{xx} n)
\Big)dx
\end{aligned}
\end{align*} By integration by parts and using the dissipation term, we get
\begin{align}\label{energy_est}
\begin{aligned}
&\frac{d}{dt}\Big(\int_\bbr\frac{|\pa_xn|^2}{2}dx
+\int_\bbr\frac{|\pa_xq|^2}{2}dx\Big)+\frac{1}{2}\int_\bbr|\pa_{xx}n|^2dx\\
&\quad\quad\leq C\int_\bbr\Big(
|\pa_x n|^2|q|^2+|\pa_x q|^2|n|^2+|\pa_x q|^2)
\Big)dx\\
&\quad\quad\leq C_{T_0}\int_\bbr\Big(
|\pa_x n|^2+|\pa_x q|^2)
\Big)dx
\end{aligned}
\end{align} where we used
\eqref{q-est_infty} and \eqref{n-est_infty}
 in the last inequality. Then by Gr\"onwall's inequality, we get
 \begin{align*}
\begin{aligned}
 \|\pa_x n\|_{L^\infty(0,T;L^2(\bbr))}
+ \|\pa_x q\|_{L^\infty(0,T;L^2(\bbr))} 
 \leq C_{T_0}\leq C.
\end{aligned}
\end{align*} In addition, by \eqref{energy_est}, we obtain
 \begin{align*}
\begin{aligned}
 \|\pa_{xx} n\|_{L^2(0,T;L^2(\bbr))}
  \leq C_{T_0}\leq C.
\end{aligned}
\end{align*}

\subsection{Conclusion}
Since  \begin{align*}
\begin{aligned}
 \|\pa_x \hatn\|_{L^\infty(0,T;L^2(\bbr)}+\|\pa_x \hat q\|_{L^\infty(0,T;L^2(\bbr)}
  \leq C,
\end{aligned}
\end{align*} we get
 \begin{align*}
\begin{aligned}
 \|\partial_x (U-\hatu)\|_{L^\infty(0,T;L^2(\bbr)}
 \leq C_{T_0}\leq C.
\end{aligned}
\end{align*}  
Hence we conclude  $$\sup_{t\in[0,T]}\|U(t) -\hatu(t)\|_{  H^1(\bbr)}\leq C.$$
Note that the above constant $C$ does not depend on  any choice of $T$ satisfying $T<T_0$,
which completes the proof.\\

\hspace{1cm}

\begin{appendix}
\setcounter{equation}{0}
\section{Proof of Lemma \ref{lem_degiorgi}} \label{app:De} 
For any constant $M>2R$, we consider a sequence $(c_k)_{k\ge 0}$ defined by
\[
c_k:=M(1-2^{-k-1}),\quad k\ge 0.
\]
Note that $M>c_{k+1}\ge c_k \ge c_0=M/2> R$ for all $k$, and $\lim_{k\to\infty} c_k = M$.\\
Let
\[
m_k:=(m-c_k)_+,
\]
and
\[
E_k:= \sup_{[0,T]} \int_\bbr m_k^2 dx + \int_0^T \int_\bbr |\pa_x m_k|^2 dxdt.
\]
Note that $E_k$ is well defined since
$$
m-R= |m|-R\leq |m|-|m_2|\leq |m-m_2|=|m_1|\in {L^\infty(0,T; L^2(\bbr))} $$ implies 
$$0\leq m_k=(m-c_k)_+\leq(m-R)_+\leq  |m_1|\in {L^\infty(0,T; L^2(\bbr))} $$
and
$$ |\pa_x m_k|=|\pa_x m \mathbf{1}_{\{m>c_k\}}|\leq |\pa_x m|\in {L^2(0,T; L^2(\bbr))}. $$
Observe that  $E_k$ is non-increasing in $k$ since 
$0\leq m_{k+1}\leq m_k$ and $|\pa_x m_{k+1}|\leq |\pa_x m_k|$ due to
$\{m>c_{k+1}\}\subset \{m>c_k\}$. We also see 
$$\int_\bbr (m-R)_+|_{t=0}dx=0\quad \mbox{and}\quad
\int_\bbr m_k|_{t=0}dx=0 \quad \mbox{for any }  k$$ due to $R\geq \|m|_{t=0}\|_{L^\infty (\bbr)}$.\\

Our goal is to show that there exists $M=M(R,T)>0$ such that
\beq\label{main_de}
\lim_{k\to\infty} E_k =0 .
\eeq
Once we prove it, then we obtain
\[
\sup_{[0, T]} \int_\bbr (m-M)_+^2 dx =0 \quad\mbox{due to}\quad 0\leq (m-M)_+\leq m_k,\quad\mbox{for any }  k,
\]
which gives the desired result.
Therefore, it remains to prove \eqref{main_de} in the following steps.\\

{\bf Step1)} Since for any constant $c$,
\[
\partial_t (m-c) - \partial_{xx} (m-c) + p_1 \partial_x (m-c) + m \partial_x(p_2+p_3)\leq 0,
\] 
$\bar m:=(m-R)_+$ satisfies
\[
\frac{d}{dt} \int_\bbr {\bar m}^2 dx +  \int_\bbr |\pa_x{\bar m}|^2 dx 
\le -\int_\bbr \bar m (\pa_x\bar m) p_1 dx -  \int_\bbr \bar m m \pa_x(p_2+p_3) dx.
\]
Then, using the integration by parts and
\[
\bar m \mathbf{1}_{\bar m>0} = (m-R) \mathbf{1}_{\bar m>0},
\]
we have
\begin{align*}
\begin{aligned}
&\frac{d}{dt} \int_\bbr {\bar m}^2 dx +  \int_\bbr |\pa_x{\bar m}|^2 dx \\
&\quad \le \int_\bbr \Big( |\bar m| |\pa_x{\bar m}| \big( |p_1| +2|p_2| +2|p_3| \big) + R |p_2| |\pa_x{\bar m}| + \bar m R |\pa_x p_3| \Big) dx,
\end{aligned}
\end{align*} 
which yields
\begin{align*}
\begin{aligned}
&\frac{d}{dt} \int_\bbr {\bar m}^2 dx + \frac{1}{2} \int_\bbr |\pa_x{\bar m}|^2 dx \\
&\quad \le C\big(\||p_1|+|p_2|+|p_3|\|_{L^\infty((0,{T} )\times \bbr)}^2
+1 \big) \|\bar m\|_{L^2(\bbr)}^2  + C R^2 \big(\|p_2\|_{L^2(\bbr)}^2 + \|\pa_x p_3\|_{L^2(\bbr)}^2  \big) \\
&\quad \le C\big(R^2
+1 \big) \|\bar m\|_{L^2(\bbr)}^2  + C R^2 \big(\|p_2\|_{L^2(\bbr)}^2 + \|\pa_x p_3\|_{L^2(\bbr)}^2  \big) .
\end{aligned}
\end{align*} 
Therefore, by the Gr\"onwall's inequality 
with \eqref{assump_lem_de} and the fact  
 $\bar m|_{t=0} =0$, there exists a positive constant $C_*=C_*(R,T_0)$ such that
\[
\sup_{[0,T]}\int_\bbr {\bar m}^2 dx + \int_0^T \int_\bbr |\pa_x{\bar m}|^2 dx dt \le C_* .
\]
This together with $c_0\ge R$ implies
\beq\label{inistep}
E_0 \le C_*.
\eeq
{\bf Step2)}  
Since for each $k\ge 1$,
\[
\frac{d}{dt} \int_\bbr |m_k|^2 dx +  \int_\bbr |\pa_x{m_k}|^2 dx 
\le -\int_\bbr m_k (\pa_x m_k) p_1 dx -  \int_\bbr m_k m \pa_x(p_2+p_3) dx,
\]
it follows from the integration by parts with $m_k \mathbf{1}_{m_k >0} = (m-c_k)\mathbf{1}_{m_k >0}$ that
\begin{align*}
\begin{aligned}
&\frac{d}{dt} \int_\bbr |m_k|^2 dx +  \int_\bbr |\pa_x{m_k}|^2 dx \\
&\quad \le \int_\bbr  |\pa_x{m_k}| \Big( |m_k|  |p_1| +   \big(2|m_k| +c_k \big)  \big(|p_2| +|p_3| \big)  \Big) \mathbf{1}_{m_k >0} dx,
\end{aligned}
\end{align*} 
which yields
\begin{align*}
\begin{aligned}
&\frac{d}{dt} \int_\bbr |m_k|^2 dx + \frac{1}{2} \int_\bbr |\pa_x{m_k}|^2 dx \\
&\quad \le C\big(\|p_1\|_{L^\infty((0,{T} )\times \bbr)}^2+\|p_2\|_{L^\infty((0,{T} )\times \bbr)}^2 +\|p_3\|_{L^\infty((0,{T} )\times \bbr)}^2 \big) \int_\bbr  \big( |m_k|^2  +c_k^2 \big) \mathbf{1}_{m_k >0}  dx .
\end{aligned}
\end{align*} 
Thus, using \eqref{assump_lem_de} and $m_k|_{t=0} =0$
\[
E_k \le CR^2 \int_0^T  \int_\bbr  \big( |m_k|^2  +c_k^2 \big) \mathbf{1}_{m_k >0}  dx dt.
\]
Note that since $m_{k-1} \ge c_k-c_{k-1} = M 2^{-k-1}$ when $m_k>0$, we have 
\[
\mathbf{1}_{m_k >0} \le M^{-1}2^{k+1} m_{k-1}  \le ( M^{-1} 2^{k+1} m_{k-1}  )^\beta,\quad \forall \beta>1 .
\] 
This together with $m_k^2 \mathbf{1}_{m_k >0} \leq  m_{k-1}^2 \mathbf{1}_{m_k >0}$ and $c_k\le M$ implies 
\[
E_k \le \frac{CR^216^k}{M^2} \int_0^T  \int_\bbr  |m_{k-1}|^4    dx dt,\quad k\ge1.
\]
But, using $\|m_k\|_{L^\infty(\bbr)} \le C \|m_k\|_{H^1(\bbr)}$ by Sobolev embedding, and
\[
\|m_k\|_{L^4((0,T)\times \bbr)}^4\le \|m_k\|_{L^\infty(0,T;L^2(\bbr))}^2 \|m_k\|_{L^2(0,T; L^\infty(\bbr))}^2,
\]
we find that 
\[
 \frac{C}{1+T}\|m_{k}\|_{L^4((0,{T} )\times \bbr)}^4 \le (E_{k})^2 ,\quad k\ge0.
\]
Therefore, there exists a positive constant $C_1=C_1(R, T_0)$ such that
\[
E_{k+1} \le \frac{C_1 16^k}{M^2} (E_k)^2,\quad\forall k\ge 0.
\]
In particular, putting $C_2=16(C_1+1)>1$,
\[
E_{k+1} \le  \frac{(C_2)^k}{M^2} (E_k)^2,\quad\forall k\ge 0.
\] 
Set $F_k:=E_k/M^2$. Then,
\[
0\leq F_{k+1} \le (C_2)^k (F_k)^2,\quad\forall k\ge 0.
\]
Moreover, since it follows from \eqref{inistep} that
\[
F_0=\frac{E_0}{M^2}\le \frac{C_*}{M^2} \to 0\quad\mbox{as } M\to\infty,
\]
using Lemma \ref{lem_va_de} below, there exists a constant $M>0$ with $M>2R$ such that
\[
\lim_{k\to\infty} F_k =0, \quad  \mbox{ so we get }\quad \lim_{k\to\infty} E_k =0 .
\]
Hence we completes the proof of Lemma \ref{lem_degiorgi}.\\

  The following lemma can be proved in a standard way (or see the proof of \cite[Lemma 1]{vasseur_de})
  \begin{lemma}\label{lem_va_de}\cite[Lemma 1]{vasseur_de}
  For $C>1$ and $\beta>1$, there exists a constant $C_0=C_0(C,\beta)$ such that for every sequence $\{W_k\}_{k=0}^\infty$ verifying $0<W_0<C_0$ and for every $k\geq 0$:
  $$0\leq W_{k+1}\leq C^k{W_k}^\beta,$$ we have
  $$\lim_{k\rightarrow \infty} W_k=0.$$
  \end{lemma}

\section{Proof of Proposition \ref{prop_lwp}} \label{app:local}
{\bf Step 1 (Iteration Scheme)} 
We first set
\[
(n^0(t,x), q^0(t,x)) = (n_0(x),q_0(x)).
\]
Then, for $k\ge 1$, and given $(n^{k-1}, q^{k-1})$, we iteratively define $(n^k, q^k)$ as a solution of the following linear system:
\begin{align}
\begin{aligned}\label{k-eq}
&\pa_t n^k =   \pa_{xx} n^k +\pa_x(n^{k-1} q^{k-1}) ,\\
&\pa_t q^k =\pa_x n^k,\\
& (n^k,q^k)|_{t=0} = (n_0,q_0).
\end{aligned}
\end{align} 
By the general theory of the heat equation together with \eqref{local:ini}, for each $k\ge 1$, if $\pa_x (n^{k-1} q^{k-1})\in L^\infty(0,T;L^2(\bbr))$ for some $\tilde{T}>0$, then 
\eqref{k-eq} has a unique solution $(n^k, q^k)$ such that
\[
(n^k-\tiln, q^k-\tilq) \in\left(C^0(0,\tilde{T};H^1(\bbr))\cap L^2(0,\tilde{T};H^2(\bbr))\right)\times C^0(0,\tilde{T};H^1(\bbr)).
\]

{\bf Step 2 (Uniform bound)} 
For convenience, we set 
\begin{align}
\begin{aligned}\label{def-NQ}
&N^k(t,x):=n^k(t,x)-\hat{n}(t,x),\quad \hat{n}(t,x):=\tiln(x-\sigma t),\\
& Q^k(t,x):=q^k(t,x)-\hat{q}(t,x), \quad\hat{q}(t,x):=\tilq(x-\sigma t).
\end{aligned}
\end{align} 
In this step, we will prove the following: for any $M_0>0$, and any initial datum $(n_0,q_0)$ satisfying \eqref{local:ini}, there exists ${T}>0$ such that
\begin{align}
\begin{aligned}\label{step2}
&\sup_{t\in[0,{T}]} \| ( N^k(t),Q^k(t)) \|_{ H^1(\mathbb{R})} \leq 2M_0,\quad \forall k\ge 0,\\
&\| \pa_x N^k \|_{ L^2(0,{T}; H^1(\mathbb{R}))} \le  2M_0,\quad \forall k\ge 1.
\end{aligned}
\end{align}

As the initial step, we first show \eqref{step2} when $k=0$. 
Since $\tiln'$, $\tilq'\in H^1(\bbr)$, and
\begin{align*}
\begin{aligned}
&\| ( N^k(t),Q^k(t)) \|_{ H^1(\mathbb{R})} \le \| N^0(t) \|_{H^1(\bbr)}+\| Q^0(t) \|_{H^1(\bbr)} \\
&\le \| n_0-\tiln \|_{H^1(\bbr)} + \|\tiln -\tiln(\cdot-\s t) \|_{H^1(\bbr)} 
 +\| q_0-\tilq \|_{H^1(\bbr)} + \|\tilq -\tilq(\cdot-\s t) \|_{H^1(\bbr)},
\end{aligned}
\end{align*} 
we use \eqref{local:ini} to have
\[
\| ( N^k(t),Q^k(t)) \|_{ H^1(\mathbb{R})}  \le M_0 + Ct,
\]
where the constant $C$ depends on $\|\tiln'\|_{H^1(\bbr)}$, $\|\tilq'\|_{H^1(\bbr)}$ and $\s$.\\
Thus, taking ${T}>0$ small enough such that $C{T}\le M_0$, we obtain \eqref{step2} when $k=0$.\\
Now, as the inductive step, for any $k\ge 1$, we assume
\beq\label{assNQ}
\sup_{t\in[0,{T}]}\| ( N^{k-1}(t),Q^{k-1}(t)) \|_{ H^1(\mathbb{R})}  \leq 2M_0.
\eeq
Since $(\hat{n},\hat{q})$ is a solution to \eqref{nq}, we use \eqref{k-eq} to find that
\begin{align}
\begin{aligned}\label{NQ}
&\pa_t N^k =   \pa_{xx} N^k +\pa_x(n^{k-1} Q^{k-1}+\hat{q}N^{k-1}) ,\\
&\pa_t Q^k =\pa_x N^k,\\
& (N^k, Q^k)|_{t=0} = (n_0-\hat{n},q_0-\hat{q}).
\end{aligned}
\end{align} 
Since it follows from \eqref{NQ} that
\begin{align*}
\begin{aligned}
&\frac{d}{dt}\int_\bbr \frac{|N^{k}|^2}{2} + \int_\bbr |\pa_xN^{k}|^2= - \int_\bbr  (n^{k-1} Q^{k-1}-\hat{q}N^{k-1}) \pa_x N^{k},\\
&\frac{d}{dt}\int_\bbr \frac{|Q^{k}|^2}{2} = \int_\bbr Q^{k} \pa_xN^k,
\end{aligned}
\end{align*} 
we use Young's inequality to have
\begin{align*}
\begin{aligned}
\frac{d}{dt}\int_\bbr \left(|N^k|^2+|Q^k|^2\right) + \int_\bbr |\pa_xN^k|^2 \le 2\int_\bbr  | n^{k-1} Q^{k-1}-\hat{q}N^{k-1}|^2 +2 \int_\bbr |Q^k|^2.\\
\end{aligned}
\end{align*} 
Since $n^{k-1}\in L^\infty([0,{T}]\times\bbr)$ by \eqref{assNQ} together with Sobolev embedding and the boundedness of $\hat{n}$, we use \eqref{assNQ} again to have
\begin{align*}
\begin{aligned}
\frac{d}{dt}\int_\bbr \left(|N^k|^2+|Q^k|^2\right) + \int_\bbr |\pa_xN^k|^2 &\le 4(\|n^{k-1}\|_{\infty}^2+\|\hat{q}\|_{\infty}^2)\int_\bbr  ( |Q^{k-1}|^2+|N^{k-1}|^2 ) +2 \int_\bbr |Q^k|^2\\
&\le C(M_0) + 2 \int_\bbr |Q^k|^2,
\end{aligned}
\end{align*} 
which implies that for some $C=C(M_0)$,
\begin{align*}
\begin{aligned}
\int_\bbr \left(|N^k(t)|^2+|Q^k(t)|^2\right) + \int_0^t  \int_\bbr |\pa_xN^k|^2  &\le e^{Ct}\int_\bbr \left(|n_0-\tiln|^2+|q_0-\tilq|^2\right) + C t e^{Ct} \\
&\le e^{Ct}  M_0^2 + C t e^{Ct}.
\end{aligned}
\end{align*} 
Thus, taking ${T}$ small again (if needed) such that $\sqrt{e^{C{T}} ( M_0^2 + C {T})} \le 2M_0$, we have
\begin{align*}
\begin{aligned}
&\sup_{t\in[0,{T}]} \| ( N^k(t),Q^k(t)) \|_{ L^2(\mathbb{R})} \leq 2M_0,\\
&\| \pa_x N^k \|_{ L^2(0,{T}; L^2(\mathbb{R}))} \le  2M_0.
\end{aligned}
\end{align*} 
Next, to estimate the higher norm, we use \eqref{NQ} to get
\begin{align*}
\begin{aligned}
&\frac{d}{dt}\int_\bbr \frac{|\pa_x N^{k}|^2}{2} + \int_\bbr |\pa_{xx}N^{k}|^2= - \int_\bbr \pa_x (n^{k-1} Q^{k-1}-\hat{q}N^{k-1}) \pa_{xx} N^{k},\\
&\frac{d}{dt}\int_\bbr \frac{|\pa_x Q^{k}|^2}{2} = \int_\bbr \pa_x Q^{k} \pa_{xx}N^k,
\end{aligned}
\end{align*} 
which gives
\begin{align*}
\begin{aligned}
\frac{d}{dt}\int_\bbr \left(|\pa_x N^k|^2+|\pa_x Q^k|^2\right) + \int_\bbr |\pa_{xx} N^k|^2 \le 2\int_\bbr  | \pa_x(n^{k-1} Q^{k-1}-\hat{q}N^{k-1})|^2 +2 \int_\bbr |\pa_x Q^k|^2.
\end{aligned}
\end{align*} 
Likewise, since \eqref{assNQ} implies
\begin{align*}
\begin{aligned}
\int_\bbr  | \pa_x(n^{k-1} Q^{k-1}-\hat{q}N^{k-1})|^2 
&\le \|\pa_xn^{k-1}\|_{L^2(\bbr)}  \|Q^{k-1}\|_{L^\infty(\bbr)} + \|n^{k-1}\|_{L^\infty(\bbr)}  \|\pa_x Q^{k-1}\|_{L^2(\bbr)} \\
& \quad +  \|\pa_x\hat{q}\|_{L^2(\bbr)}  \|N^{k-1}\|_{L^\infty(\bbr)} + \|\hat{q}\|_{L^\infty(\bbr)}  \|\pa_x N^{k-1}\|_{L^2(\bbr)} \\
& \le C(M_0),
\end{aligned}
\end{align*} 
we have that (for ${T}$ smaller than above if needed)
\begin{align*}
\begin{aligned}
&\sup_{t\in[0,{T}]} \| (\pa_x N^k(t),\pa_xQ^k(t)) \|_{ L^2(\mathbb{R})} \leq 2M_0,\\
&\| \pa_{xx} N^k \|_{ L^2(0,{T}; L^2(\mathbb{R}))} \le  2M_0.
\end{aligned}
\end{align*} 
Therefore, we have \eqref{step2}.\\

{\bf Step 3 (Uniform bound for $1/n$)} 
Since it follows from \eqref{step2} and from Sobolev embedding  that for all $k\ge 1$,
\begin{align*}
\begin{aligned}
&\pa_x(n^{k-1}q^{k-1})= (\pa_x n^{k-1}) q^{k-1}+ n^{k-1} \pa_x q^{k-1}\\
&~=(\pa_x n^{k-1}) (q^{k-1}-\hat{q}) + (\pa_x n^{k-1})\hat{q}+ (n^{k-1}-\hat{n}) \pa_x q^{k-1}+ \hat{n}\pa_x q^{k-1}
\in L^\infty(0,T;L^2(\bbr)),
\end{aligned}
\end{align*} 
and $\|\pa_x(n^{k-1}q^{k-1})\|_{L^\infty(0,T; L^2(\bbr))} \le C(M_0)$,
we use Duhamel's principle to represent
$$n^k(t,x)=\int_\bbr\Phi(t,x-y)n_0(y)dy+\int_0^t\int_\bbr 
\Phi(t-s,x-y)[\pa_x(n^{k-1}q^{k-1})](s,y),
dy\,ds$$
where  $\Phi(t,x)=\frac{1}{\sqrt{4\pi t}}e^{-\frac{|x|^2}{4t}}$ is the heat kernel in 1D.\\
Thus, by Young's inequality, we have the following estimate: for all $x\in\bbr,~ t\in[0,T]$,
\begin{align*}
\begin{aligned}
n^k(t,x)&\geq \inf_\bbr n_0  - \int_0^t \|\Phi(t-s, \cdot)\|_{L^2(\bbr)}\cdot
\|[\pa_x(n^{k-1}q^{k-1})](s,\cdot)\|_{L^2(\bbr)}ds\\
&\geq r_0  - \|[\pa_x(n^{k-1}q^{k-1})]\|_{L^\infty(0,{T}; L^2(\bbr))} \int_0^T \|\Phi(s,\cdot)\|_{L^2(\bbr)}ds\\
&\geq  r_0- C\cdot\|[\pa_x(n^{k-1}q^{k-1})]\|_{L^\infty(0,{T}; L^2(\bbr))}  \cdot T^{3/4}
\ge r_0 -C(M_0) T^{3/4}.
\end{aligned}
\end{align*}  
Therefore, taking ${T}$ small again (if needed) such that $r_0 -C(M_0) T^{3/4}\ge r_0/2$, we have
\beq\label{k-low}
\inf_{t\in[0,T]} \inf_{x\in\bbr} n^k(x,t) \ge \frac{r_0}{2}.
\eeq

{\bf Step 4 (Convergence)} 
We will first prove that the sequence $\{(N^k,Q^k)\}_{k\ge1}$ is Cauchy in $\mathcal{S}$, where
\[
\mathcal{S}:=\left(L^\infty(0,{T};L^2(\bbr))\cap L^2(0,{T};H^1(\bbr))\right)\times L^\infty(0,{T};L^2(\bbr)).
\]
For convenience, we set 
\[
\bar N^k := N^{k+1}-N^k,\qquad  \bar Q^k := Q^{k+1}-Q^k.
\]
Then, using $\bar N^{k-1}=n^k-n^{k-1}$ and $\bar Q^{k-1}=q^k-q^{k-1}$, it follows from \eqref{NQ} that
\begin{align}
\begin{aligned}\label{barNQ}
&\pa_t \bar N^k =   \pa_{xx} \bar N^k +\pa_x(n^{k} \bar Q^{k-1} + q^{k-1} \bar N^{k-1}) ,\\
&\pa_t \bar Q^k =\pa_x \bar N^k,\\
& (\bar N^k, \bar Q^k)|_{t=0} = (0,0).
\end{aligned}
\end{align} 
Thus, using the same estimate as in Step 2, we find that for all $t\le {T}$,
\begin{align*}
\begin{aligned}
&\frac{d}{dt}\int_\bbr \left(|\bar  N^k|^2+|\bar Q^k|^2\right) + \int_\bbr |\pa_x\bar N^k|^2 \\
&\le 4(\|n^{k}\|_{L^\infty([0,{T}]\times\bbr)}^2+\|q^{k-1}\|_{L^\infty([0,{T}]\times\bbr)}^2)\int_\bbr  ( |\bar Q^{k-1}|^2+|\bar N^{k-1}|^2 ) +2 \int_\bbr |\bar Q^k|^2 .
\end{aligned}
\end{align*}  
Using the uniform-in-$k$ bound \eqref{step2} with Sobolev embedding, we have
\[
\frac{d}{dt}\int_\bbr \left(|\bar  N^k|^2+|\bar Q^k|^2\right) + \int_\bbr |\pa_x\bar N^k|^2 \le C(M_0) \int_\bbr  ( |\bar Q^{k-1}|^2+|\bar N^{k-1}|^2 ) +2 \int_\bbr |\bar Q^k|^2 .
\]
Integrating it over $[0,{T}]$, we have
\[
\int_\bbr \left(|\bar  N^k(t)|^2+|\bar Q^k(t)|^2\right)  + \int_0^t \int_\bbr |\pa_x\bar N^k|^2 \le C \int_0^t \Big(\int_\bbr  ( |\bar Q^{k-1}|^2+|\bar N^{k-1}|^2 ) +2 \int_\bbr |\bar Q^k|^2\Big) .
\]
This implies
\[
\int_\bbr \left(|\bar  N^k(t)|^2+|\bar Q^k(t)|^2\right)  + \int_0^t \int_\bbr |\pa_x\bar N^k|^2 \le \frac{C t^k}{k!},\quad \forall t\le {T} .
\]
Therefore, the sequence $\{(N^k,Q^k)\}_{k\ge1}$ is Cauchy in $\mathcal{S}$, which implies that
 there exists a limit $(N^\infty, Q^\infty)$ such that 
\beq\label{strong-con}
\mbox{$(N^k,Q^k) \to (N^\infty, Q^\infty)\quad$ in $\mathcal{S}$.}
\eeq
Furthermore, using \eqref{strong-con} and \eqref{step2}, we have
\begin{align}
\begin{aligned}\label{wlsc}
\sup_{t\in[0,{T}]} \| ( N^\infty(t),Q^\infty(t)) \|_{ H^1(\mathbb{R})} \leq 2M_0,\qquad \| \pa_x N^\infty \|_{ L^2(0,{T}; H^1(\mathbb{R}))} \le  2M_0.
\end{aligned}
\end{align} 

{\bf Step 5 (Existence)} 
Let $n:=N^\infty +\hat{n}$ and $q:=Q^\infty +\hat{q}$. Then, by \eqref{def-NQ} and \eqref{strong-con}, we obtain that 
\beq\label{skcon}
\mbox{$(n^k-\hat n,q^k-\hat q) \to (n-\hat n,q-\hat q)\quad $ in $\mathcal{S}$} ,
\eeq
and 
\beq\label{final-lsc}
 \| (n-\hat{n}, q-\hat{q}) \|_{L^\infty(0,{T};H^1(\mathbb{R}))} \leq 2M_0,\qquad  \pa_x n \in L^2(0,{T}; H^1(\mathbb{R})).
\eeq
This implies $n^k-\hat n \to n-\hat n$ in $L^\infty(0,{T};H^{3/4}(\mathbb{R}))$, and thus $n^k\to n$ in $L^\infty(0,{T};L^\infty(\mathbb{R}))$, which together with \eqref{k-low} yields
\[
\inf_{t\in[0,T]} \inf_{x\in\bbr} n(x,t) \ge \frac{r_0}{2}.
\]
Moreover, \eqref{skcon} and \eqref{final-lsc} together with \eqref{k-eq} imply that $(n,q)$ solves \eqref{nq} with $(n,q)|_{t=0}=(n_0,q_0)$ in the sense of distributions. In particular, \eqref{nq} and \eqref{final-lsc}  yield that $\partial_t(n-\tiln) \in L^2(0,{T};H^2(\bbr))$, which together with Aubin-Lions lemma implies $n-\tiln \in C([0,{T}];H^1(\bbr))$, and thus $q-\tilq \in C([0,{T}];H^1(\bbr))$. \\

{\bf Step 6 (Uniqueness)} 
Let $(n_1,q_1)$ and $(n_2,q_2)$ be solutions to \eqref{nq} with the initial datum $(n_0,q_0)$, and satisfy \eqref{final-lsc}. Then, set
$\bar n := n_1 -n_2,\quad \bar q := q_1 -q_2$. \\
Then, it follows from \eqref{nq} that
\begin{align*}
\begin{aligned}
&\pa_t \bar n =   \pa_{xx} \bar n +\pa_x(n_1 \bar q + q_2 \bar n) ,\\
&\pa_t \bar q =\pa_x \bar n,\\
& (\bar n, \bar q)|_{t=0} = (0,0).
\end{aligned}
\end{align*}
which has the same structure as in \eqref{barNQ}. Thus, using the same estimates as above, we have
\[
\int_\bbr \left(|\bar n(t)|^2+|\bar q(t)|^2\right)  \le C \int_0^{{T}} \int_\bbr \left(|\bar n(t)|^2+|\bar q(t)|^2\right) , \quad \forall t\le {T},
\]
which implies that $n_1 = n_2$ and $q_1 = q_2$ on $[0,{T}]\times\bbr$. \\

\end{appendix}

\bibliography{ckv2019bib}
 
\end{document}